\documentclass[12pt]{amsart}
\usepackage[normalem]{ulem}
\usepackage[utf8]{inputenc}
\usepackage{enumitem,color, amssymb,fullpage}
\usepackage[all,cmtip]{xy}
\usepackage{tikz,epsfig}
\usepackage{verbatim}
\usepackage{hyperref}
\usepackage{geometry}
\usepackage{colonequals}
\usepackage{tabularray}
\usepackage{cleveref}
\usepackage{caption} 
\usepackage{tikz}
\usepackage{tikz-cd}

\definecolor{purple}{rgb}{0.59, 0.44, 0.84}

\definecolor{orange}{rgb}{1, 0.6, 0.4}

\newcommand{\Mod}[1]{\ (\mathrm{mod}\ #1)}

\newtheorem{theorem}{Theorem}[section]
\newtheorem{lemma}[theorem]{Lemma}
\newtheorem{corollary}[theorem]{Corollary}
\newtheorem{proposition}[theorem]{Proposition}

\newtheorem{defn}{Definition}[section]

\newcounter{example}
\newenvironment{example}[1][]{\refstepcounter{example}\par\medskip
   \noindent \textbf{Example~\theexample. #1} \rmfamily}{\medskip}
\newcounter{remark}
\newenvironment{remark}[1][]{\refstepcounter{remark}\par\medskip
   \noindent \textbf{Remark~\theremark. #1} \rmfamily}{\medskip}

\numberwithin{equation}{section}
\numberwithin{example}{section}
\numberwithin{remark}{section}

\newcommand{\ba}{\backslash}

\def \l {\lambda}

\def\P{\mathbb P}
\def\C{\mathbb C}
\def\G{\mathbf G}

\def\g{\gamma}

\def\GL{\mathrm{GL}}

\def\Tr{{\rm Tr}}

\def\Frob{{\rm Fr}}
\def\F{{\mathbb F}}
  
\def\Z{{\mathbb Z}}

\def\Q{{\mathbb Q}}

\def\wp{\mathfrak{p}}

\def\BK{\mathbb{K}}

\def\G{\Gamma}

\def\SL{\mathrm{SL}}

\def\M#1#2#3#4{\begin{pmatrix}\, #1&#2\\#3&#4\end{pmatrix}}

\def\N{\mathbb N}

\def\C{\mathbb{C}}

\def\Z{\mathbb{Z}}
\def\Q{\mathbb{Q}}
\def\F{\mathbb{F}}

\def\M#1#2#3#4{\begin{pmatrix}#1&#2\\#3&#4\end{pmatrix}}

\def \Frob{\text{Frob}}
\def \ol{\overline}
\def \eps{\varepsilon}

\newcommand{\fp}{\mathbb{F}_p}

\def \g {\mathfrak{g}}

\newcommand{\fq}{\mathbb{F}_q}

\newcommand{\kphat}
{\widehat{\kappa_{\wp}^{\times}}}

\newcommand{\hypcoeff}[1]{\frac{(\ba)_{#1}}{(\bbeta)_{#1}}}

\newcommand*\HYPERskip{&}
\catcode`,\active
\newcommand*\pFq{
\begingroup
\catcode`\,\active
\def ,{\HYPERskip}%
\doHyper
}
\catcode`\,12
\def\doHyper#1#2#3#4#5{%
\, _{#1}F_{#2}\left[\begin{matrix}#3 \smallskip \\  #4\end{matrix} \; ; \; #5\right]%
\endgroup
}

\newenvironment{singnumalign}{
    \begin{equation}
    \begin{aligned}
}{
    \end{aligned}
    \end{equation}
    \ignorespacesafterend
}

\catcode`,\active

\catcode`\,12

\crefformat{section}{\S#2#1#3}
\crefformat{subsection}{\S#2#1#3}

\catcode`,\active

\catcode`\,12

\newcommand*\HYPERpp{&}
\catcode`,\active
\newcommand*\pPPq{
\begingroup
\catcode`\,\active
\def ,{\HYPERpp}%
\doHyperFpp
}
\catcode`\,12
\def\doHyperFpp#1#2#3#4#5{%
\, _{#1}{\mathbb P}_{#2}\left[\begin{matrix}#3 \smallskip \\  #4\end{matrix} \; ; \; #5\right]%
\endgroup
}

\def\SL{\mathrm{SL}}

\def\M#1#2#3#4{\begin{pmatrix}#1&#2\\#3&#4\end{pmatrix}}

\def\N{\mathbb N}

\def\C{\mathbb{C}}

\def\Z{\mathbb{Z}}
\def\Q{\mathbb{Q}}
\def\F{\mathbb{F}}

\def\M#1#2#3#4{\begin{pmatrix}#1&#2\\#3&#4\end{pmatrix}}

\def \Frob{\text{Frob}}
\def \ol{\overline}
\def \eps{\varepsilon}
\def \HD{\mathit{HD}}
\def \ba{\boldsymbol{\alpha}}
\def \bbeta{\boldsymbol{\beta}}

\begin{document}

\title{The explicit hypergeometric-modularity  method I}

\author{Michael Allen, Brian Grove, Ling Long, Fang-Ting Tu}
\begin{abstract}
    The theories of hypergeometric functions and modular forms are highly intertwined.  For example, particular values of truncated hypergeometric functions and hypergeometric character sums are often congruent or equal to Fourier coefficients of modular forms.  In this series of papers, we develop and explore an explicit ``Hypergeometric-Modularity" method for associating a modular form to a given hypergeometric datum.  In particular, for certain length three and four hypergeometric data we give an explicit method for finding a modular form $f$ such that the corresponding hypergeometric Galois representation has a subrepresentation isomorphic to the Deligne representation of $f$.  Our method utilizes Ramanujan's theory of elliptic functions to alternative bases, commutative formal group laws, and supercongruences.  As a byproduct, we give a collection of eta quotients with multiplicative coefficients constructed from hypergeometric functions. In the second paper, we discuss a number of applications, including explicit connections between hypergeometric values and periods of these explicit eta quotients as well as evaluation formulae for certain special $L$-values.
\end{abstract}
\maketitle
\tableofcontents

\section{Introduction}\label{ss:intro} 

The Langlands program predicts motivic Galois representations are modular or automorphic. For example, the Galois representation arising from any elliptic curve $E$ defined over $\Q$ is modular. This modularity enables the analytic continuation of the $L$-function of $E$, and consequently the well-definedness of the central $L$-value of $E$---a vital component of the influential Birch and Swinnerton-Dyer conjecture. Among the motivic Galois representations, there is a category consisting of hypergeometric Galois representations.  These hypergeometric Galois representations cover a wide spectrum of cases and can be computed explicitly. Thus, they form a desirable testing ground for experimental results towards open conjectures as in \cite{RRV22}, as well as for developing new theories.  By their motivic nature, these representations can be realized across various arithmetic settings.  Over $\C$, classical hypergeometric functions can be used to evaluate $L$-values and periods of hypergeometric varieties \cite{DPVZ, DKSSVW, LLT}.  Hypergeometric character sums over finite fields yield point-counting formulae \cite{BCM, Win3X, Greene87, LLT2}.  Finally, $p$-adic hypergeometric functions contain deep structural information via supercongruences and Dwork's $p$-adic unit root theory \cite{Allen, Dwork, LTYZ, Mortenson-wt3, Mortenson08, OSZ, Swisher15}.  
We explore these connections throughout this series of papers.
In this first paper we develop our general method for associating modular forms to hypergeometric Galois representations.  Further applications and examples are given in the second paper \cite{HMM2}.

To continue, we first recall the basic notation. Let $a \in \C$ and $n \in \N$. Define the rising factorial $(a)_{n}$ as $(a)_{n} = a(a+1) \cdots (a+n-1)$ with $(a)_{0}=1$. The classical (generalized) $_{n}F_{n-1}$ hypergeometric functions with complex parameters given by multi-sets $\ba \colonequals \left\{r_{1}, \ldots, r_{n}\right\}$ and $\bbeta \colonequals \left\{q_{1} = 1, q_{2}, \ldots, q_{n}\right\}$ and argument $z$ are defined as 
$$\pFq{n}{n-1}{r_{1}&r_{2}&\cdots&r_{n}}{1&q_2&\cdots&q_{n}}{z} \colonequals \sum_{k=0}^{\infty} \frac{ (r_{1})_{k} \cdots (r_{n})_{k}}{(q_{2})_{k} \cdots (q_{n})_{k}} \frac{z^k}{k!}$$
and converge when $|z|<1$.  We note that we take the nonstandard convention to include the extra parameter $q_1=1$ in $\bbeta$.  This parameter corresponds to the $k! = (1)_k$ term on the right-hand side, and so this notation will allow us to better keep track of the contribution of $k!$ to the coefficients of our hypergeometric series.
\par  
Together, we call $\HD = \left\{\ba, \bbeta\right\}$ a \emph{hypergeometric datum} of length $n$.  Throughout we assume the parameters $r_i$ and $q_j$ are rational.  We use the shorthand notations $F(\ba, \bbeta;z)$, $F\left(\begin{array}{c}\ba\\ \bbeta\end{array};z\right)$, or simply $F(\HD; z)$ to denote the corresponding ${}_nF_{n-1}$ function and $F(\ba,\bbeta;z)_{m-1}$ or $F(\HD;z)_{m-1}$ for its truncation after $m$ terms. For a fixed datum $\HD$, we use $M(HD)$ to denote the least positive common denominator of the $r_i$ and $q_j$.

We illustrate the relationship between hypergeometric functions and modular forms using the example $\HD = \left\{\left\{ \frac{1}{2}, \frac{1}{2}, \frac{1}{2}, \frac{1}{2}\right\}, \left\{1, 1, 1, 1\right\}\right\}$.  In \cite{AO00}, Ahlgren and Ono showed
\begin{equation}\label{eq:AO}
H_{p}\left[\begin{matrix}\, \frac{1}{2}&\frac{1}{2}&\frac{1}{2}&\frac{1}{2} \smallskip \\ \, \, 1&1&1&1 \end{matrix} \; ; \; 1 \right] = a_{p}(f_{8.4.a.a})+p,
\end{equation}
for all odd primes $p$, where the $H_{p}$ function is a finite hypergeometric function defined in \cite{McCarthy} and recalled in \eqref{eq:H}, $a_p(f)$ denotes the $p$-th Fourier coefficient of $f$, and the subscript $8.4.a.a$ is the LMFDB \cite{LMFDB} label of the modular form $f$.
In \cite{Kilbourn} Kilbourn proved the following supercongruence for odd primes $p$:
\begin{equation}\label{eq:Kilbourn}
 \pFq43{\,\frac{1}{2}&\frac{1}{2}&\frac{1}{2}&\frac{1}{2} }{ \,\, 1&1&1&1 }1 _{p-1}\equiv a_{p}(f_{8.4.a.a})\pmod{p^3}.   
\end{equation}
Later, Zagier  in \cite{Zagier-top-diff} obtained the following special $L$-value of $f_{8.4.a.a}$ at 2:
\begin{equation}\label{eq:zagier}
  L \left(f_{8.4.a.a}, 2 \right) = \frac{\pi^{2}}{16} \pFq{4}{3}{\,\frac{1}{2}&\frac{1}{2}&\frac{1}{2}&\frac{1}{2}}{\, \,1&1&1&1}{1}.  
\end{equation} 
In the literature, many such modularity results are predicted by modularity--lifting theorems.  However, identifying the precise automorphic targets often results in other useful information. The main result of this paper is an explicit method by which to find the Hecke eigenform corresponding to a given hypergeometric Galois representation as in the previous example.  Using the inductive integral definition for hypergeometric functions (see \cref{eq:inductive}), we derive a modular form $f_{\HD}$ by specializing the parameter of the integrand to be a suitable modular function.  For example,
\begin{equation}
    \pFq{4}{3}{\frac{1}{2}&\frac{1}{2}&\frac{1}{2}&\frac{1}{2}}{1&1&1&1}{1} = \frac{1}{\pi} \int_0^1 \left(\frac{t}{1-t}\right)^{1/2} \pFq{3}{2}{\frac{1}{2}&\frac{1}{2}&\frac{1}{2}}{1&1&1}{t} \frac{dt}{t}.
\end{equation}
The modular form $f_{8.4.a.a}$ then arises from
\begin{equation}
    \label{eq:HG-f8.4.a.a}
   f_{8.4.a.a}(q^{1/2})=\frac 18 \left(\left(\frac{t}{t-1}\right)^{\frac12}\pFq32{\frac12&\frac12&\frac12}{1&1&1}t q\frac{dt}{t{dq}}\right)_{t=-64\left(\frac{\eta(q^2)}{\eta(q)}\right)^{24}},
\end{equation}
where $\eta(q)=q^{1/24}\prod_{n\ge 1}(1-q^n)$ denotes the Dedekind-eta function. When $q$ is specialized to $e^{2\pi i\tau}$ we use $\eta(\tau)$ for $\eta(q)$.
\Cref{thm:main} generalizes this approach to many more length $3$ and $4$ data.  For the reader's benefit, we state a specialization of \Cref{thm:main} for an illustrative family of examples.  Set $\HD(r,s) = \left\{\left\{\frac{1}{2}, \frac{1}{2}, r\right\}, \left\{1, 1, s\right\}\right\}$ for any of the 167 pairs $(r,s)$ in the set
\begin{equation}\label{eq:S}
    \begin{split}
         {\mathbb S}_2 \colonequals \{(r, s)\mid 0<r<s<\frac32, r\neq 1, s\neq \frac12, \, 24s\in \Z\, , \, 8(r+s)\in\Z\}.
    \end{split}
\end{equation}
The Euler Integral formula \eqref{eq:inductive} gives
\[
    F(\HD(r,s); 1) = \frac{1}{B(r, s-r)} \int_0^1 t^r(1-t)^{s-r-1} \pFq{2}{1}{\frac{1}{2} & \frac{1}{2}}{1 & 1}{t} \frac{dt}{t},
\]
where $B(r,s)$ is the beta function, see \cite{AAR}.  For these examples, we let $t$ be the modular lambda function $\lambda(\tau)$ and normalize, which yields the modular form
\begin{align*}
    \BK_2(r,s)(\tau) &\colonequals 2^{1-4r} \l^{r-1}(1-\l)^{s-r-1}\pFq21{\frac12&\frac12}{1&1}{\l} q\frac{d\l}{dq} \\
    &= \frac{\eta \left(\frac{\tau}{2} \right)^{16s-8r-12}\eta(2 \tau)^{8s+8r-12}}{\eta(\tau)^{24s-30}}.
\end{align*}
Setting $N(r) = 48/\gcd(24r, 24)$---twice the denominator of $r$ when written in lowest terms---each $\BK_2(r,s)(N(r)\tau)$ is a weight-three cusp form of level $\frac{48}{N(r)}\frac{48}{N(s-r)}$, see \Cref{ss:K2}.  For a given $(r,s)\in {\mathbb S}_2$ and  $c\in\mathbb Z_{>0}$, let  $(r_c,s_c)\equiv c(r,s) \mod \Z$ be the corresponding conjugate pair in $\mathbb S_2$. That is, let
$r_c$,
be the fractional part of $cr$, which we denote by $\left\{cr\right\}$, and let $s_c=\{cs\}$  if $\{cr\}-{\{cs\}}\le 0$ and $s_c=\{cs\}+1$ otherwise.  For a number field $K$, let $G_K\colonequals\text{Gal}(\overline \Q/K)$ and in particular let $G(M)$ denote $G_{\Q(\zeta_{M})}$. In this setting, our main result \Cref{thm:main} specializes as follows:
\begin{theorem}\label{thm:mainK2version}Given  $(r,s)\in {\mathbb S}_2$, let  $\HD(r,s):=\{\{\frac 12, \frac 12, r\}, \{1,1,s\}\}$, $M=M(\HD(r,s))$, and  $\rho_{\{\HD(r,s);1\}}:G(M)\rightarrow GL_2(\overline \Q_\ell)$ be the associated Katz representation   (cf \Cref{thm:Katz}).  If for all $c\in(\Z/M\Z)^\times$, $(r_c,s_c)\in {\mathbb S}_2 $ and $\BK_2( r_c,s_c)(N(r)\tau)$ lie in the same Hecke orbit, then there is an explicit finite character $ \chi_r$ of $G(M)$   depending on $r$ such that \begin{equation}\label{eq:1.5}
     \left( \chi_r\otimes \rho_{\{\HD(r,s);1\}}\right)|_{G(2M)}\simeq \rho_{f_{\HD(r,s)}^\sharp}|_{G(2M)},
    \end{equation}where $\rho_{f_{\HD(r,s)}^\sharp}$ denotes the Deligne representation of $G_\Q$ associated with a Hecke eigenform $f_{\HD(r,s)}^\sharp$  which can be expressed as 
\begin{equation}
      f_{\HD(r,s)}^\sharp(\tau)=\sum_{c\in (\Z/M\Z)^\times} t_c \cdot \BK_2( r_c,s_c)(N(r)\tau)
\end{equation}where $t_c\in\overline \Q$ are  determined by the first few coefficients of $\BK_2( r,s)(N(r)\tau)$.
\end{theorem}

For example, when $(r,s)=(\frac12,1)$, we have $M(HD)=2$ and $f^{\sharp}_{HD(\frac12,1)}(\tau) = f_{16.3.c.a}(\tau) = \BK_2(\frac12,1)(4 \tau)$. In this case, as  representations of $G_\Q$, $\rho_{\{\HD(\frac12,1);1\}}\simeq \rho_{f_{\HD(\frac12,1)}^\sharp}$. This recovers work of Ahlgren, Ono, and Penniston \cite{AOP}. 
An example 
for a Hecke orbit of size two is $(r,s)=(\frac14,\frac34)$, for which
\begin{equation}\label{eq:K2(1/4,3/4)}
    f^{\sharp}_{\HD(\frac14,\frac34)}(\tau) = f_{32.3.c.a}(\tau) = \BK_2\left(\frac14,\frac34\right)(8 \tau)+4 \sqrt{-1} \BK_2\left(\frac34,\frac54\right)(8 \tau).
\end{equation}
The explicit construction of the $\mathbb{K}_{2}$ functions implies that the  $L$-value corresponding to \eqref{eq:K2(1/4,3/4)} is

\begin{equation*}
\begin{split}
&L(f_{32.3.c.a},1)\\
&= \frac{1}{8} \left(B\left(\frac{1}{4},\frac{1}{2}\right) \cdot \, \pFq{3}{2}{\frac{1}{2}&\frac{1}{2}&\frac{1}{4}}{&1&\frac{3}{4}}{1} + \sqrt{-1} B\left(\frac{3}{4},\frac{1}{2}\right) \cdot \, \pFq{3}{2}{\frac{1}{2}&\frac{1}{2}&\frac{3}{4}}{&1&\frac{5}{4}}{1} \right).
\end{split}
\end{equation*}

\begin{remark}
    \Cref{thm:mainK2version}  can be alternatively stated in terms of finite hypergeometric sums and Fourier coefficients, see \Cref{thm:main} for this formulation.  Further, we note that the restriction of the representations in \eqref{eq:1.5} to $G(2M)$ instead of $G(M)$ is not necessary, but yields a clearer result by removing a sign that appears in \Cref{thm:main}.
    A complete list of $(r,s) \in \mathbb{S}_2$ satisfying the hypotheses of \Cref{thm:mainK2version} can be found in \cite[Table 1]{ENRosenK2} by Rosen. 
\end{remark}

Our main \Cref{thm:main} is more general than \Cref{thm:mainK2version} and yields many applications, see \cite{HMM2, DMgrove, ENRosenK1}.
Our method is based on Ramanujan's theories of elliptic functions to alternative bases (REAB) \Cref{subsec: Ramanujan},  the theory of commutative formal group laws (CFGL) \Cref{ss:CFGL}, hypergeometric character sums and Galois representations \Cref{sec:character sums}, and the residue-sum technique for proving supercongruences \Cref{sec:super}, as in \cite{Allen,LTYZ}.
\subsection*{Acknowledgements} This paper is inspired by the work of Beukers in \cite{Beukers-another} and the authors' earlier discussion with Zudilin.  The authors would like to express their sincere gratitude to Frits Beukers, Henri Darmon, Wen-Ching Winnie Li, Tong Liu, Dermot McCarthy, Esme Rosen, Hasan Saad, Armin Straub, John Voight, Pengcheng Zhang, and Wadim Zudilin for productive conversations and feedback on this project.  Finally, the authors thank the anonymous referees for their thoughtful comments that have led to a much-improved paper.
\par
Grove is partially supported by a summer research assistantship from the Louisiana State University (LSU) Department of Mathematics. Long is supported in part by the Simons Foundation grant  \#MP-TSM-00002492 and the LSU Michael F. and Roberta Nesbit McDonald Professorship; Tu is supported by the NSF grant DMS \#2302531.
\section{Statement of Results}\label{s:2}
\subsection{Notation and Hypergeometric Functions}
Consider the hypergeometric datum 
\[
    \textit{HD}= \{\ba = \{r_{1}, \ldots, r_{n}\}, \bbeta = \{q_1, \ldots, q_{n}\}\}.
\]
Throughout the paper we will assume that all $r_i, q_j \in \Q$ and at least two of the $q_j$ are equal to $1$. We refer to $n$ as the length of the datum.  
We say the pair $(\ba, \bbeta)$ is \emph{primitive} if $r_{i}-q_{j} \notin \Z$ for all $1 \leq i,j \leq n$. The primitive assumption is to make sure the corresponding local system is irreducible, as described in \cite{Beukers-Heckman} by Beukers--Heckman. \par
\subsubsection{Euler's integral representation}
 The discussion of classical hypergeometric functions so far has involved formal power series in $z$ which satisfy a certain Fuchsian differential equation. An alternate perspective on classical hypergeometric functions is from the integral representation of Euler. When $\text{Re}(q_{i})>\text{Re}(r_{i})>0$, there is an inductive formula to construct hypergeometric functions, see \cite[(2.2.2)]{AAR}. {Define $_{1}F_{0}[r_{1};z]\colonequals (1-z)^{-r_{1}}$. Then for $n \geq 2$}
\begin{multline}\label{eq:inductive}
    \pFq{n}{n-1}{\, \, r_{1}&r_{2}&\cdots&r_{n}}{1&q_{2}&\cdots&q_{n}&}{z}=\\ \frac{\G(q_{n})}{\G(r_{n})\G(q_{n}-r_{n})} \int_0^1 t^{r_{n}}(1-t)^{q_{n}-r_{n}-1}  \pFq{n-1}{n-2}{\, \,r_{1}&r_{2}&\cdots&r_{n-1}}{1&q_{2}&\cdots&q_{n-1}}{tz}\frac{dt}{t},
\end{multline}
where $\Gamma(x)$ is the usual gamma function. As we demonstrated in the introduction, this integral representation plays a key role in our method for associating a modular form to a hypergeometric datum.

\subsubsection{Hypergeometric character sums}\label{ss:2.1.4}
In \cite{Win3X}, a character sum $_{n}\mathbb{P}_{n-1}$ is defined inductively, parallel to \eqref{eq:inductive}. Passing from the ring of cyclotomic integers $\Z[\zeta_M]$---where $\zeta_M$ is a primitive $M^{th}$ root of unity---to the finite field setting is done as follows. For each nonzero prime ideal $\wp$ coprime to $M$ and integer $i$ we can associate to the residue field $\kappa_\wp\colonequals\Z[\zeta_M]/\wp$ of size $q$ a character using the $M^{th}$ residue symbol.  We set
\begin{equation}\label{eqn:residue symbol}
    \iota_\wp\left(\frac iM\right)(x) \colonequals \left( \frac x \wp\right)_M^i\equiv x^{(q-1)\frac{i}M} \pmod {\wp}, \quad  \forall x\in \Z[\zeta_M].
\end{equation} 
For example, if $\wp$ is coprime to $2$, $\iota_\wp(1/2)=\phi_\wp$ or simply $\phi$, the quadratic character of the residue field.  Likewise, $\iota_\wp(1)$ is the trivial character, which we denote by $\eps_\wp$ or simply $\eps$.  At times we will want to place the values of this character in $\C_p$, in which case we take  $\iota_\wp\left(\frac 1{q-1}\right)$ to be $\bar{\omega}_q$, where $\omega_q$ is the Teichm\"{u}ller character.  For simplicity, we write $R_i$ for $\iota_\wp(r_i)$, $\overline{R}_i$ for $\iota_\wp(-r_i)$, and define $Q_i$ and $\overline{Q}_i$ analogously.  A finite character sum is obtained by replacing $t^{r_{i}}$ in the complex setting with the character value $ R_{i}(t)$  for $t \in \fq$,  and replacing gamma functions with Gauss sums. 
To be more explicit, for a fixed $\l \in \kappa_\wp^\times$, define
$$
  {}_{1}\mathbb{P}_{0}[R_{1};\lambda;q] \colonequals \ol{R_1}(1-\l).
$$
Then for $n \geq 2$,
  \begin{multline}
    \label{eq:P-by-induction}
   \P(\HD;\l;\wp) = \pPPq{n}{n-1}{R_1& R_2&\cdots &R_{n}}{\eps& Q_2&\cdots &Q_{n}}{\l;q} \\
   \colonequals \sum_{x\in{\kappa_\wp}} R_{n}(x)\overline R_{n} Q_{n}(1-x) \cdot
   \pPPq{n-1}{n-2}{R_1& R_2&\cdots &R_{n-1}}{\eps & Q_2&\cdots &Q_{n-1}}{\l x;q}\\
   = \frac{(-1)^{n}}{q-1} \left(\prod_{i=2}^{n}R_{i}Q_{i}(-1) \right) \sum_{\chi \in \kphat} \binom{R_{1}\chi}{\chi} \binom{R_{2}\chi}{Q_{2}\chi} \cdots \binom{R_{n}\chi}{Q_{n}\chi}\chi(\l),\\ 
\end{multline} 
where $\kphat$ is the character group for $\kappa_{\wp}^{\times}$, $\chi(0)$ is set to be $0$ for any $\chi \in \kphat$, and 
$$\binom{A}{B} \colonequals -B(-1)J(A, \overline{B}) = -B(-1) \sum _{x\in \kappa_\wp} A(x)\overline{B}(1-x),$$
for any characters $A,B$  and  $J(A,B)$ is a Jacobi sum. 
When the length $n$ pair $(\ba,\bbeta)$ is primitive, let
\begin{equation}\label{eq:P-H}
H_q(\ba,\bbeta;\l;\wp)  :  ={(-1)^{n-1}}\mathcal J(\HD;\wp)^{-1} \P(\ba,\bbeta;\l;\wp), 
 \end{equation} where
 \begin{equation}\label{eq:calJ}
  \mathcal J(\HD;\wp)\colonequals   \prod_{i=2}^n -J(\iota_\wp(r_i),\iota_\wp(q_i-r_i)).
 \end{equation}
 If $\bbeta=\{1,\cdots,1\}$, then 
 \begin{equation} \label{eq:nf}
     \mathcal J(\HD;\wp)=\prod_{i=2}^n\iota_\wp(r_i)(-1).
 \end{equation}  If the value $H_q(\ba,\bbeta;\l;\wp)$ is an integer, then it is independent of the choice of generator and so we omit $\omega_p$ from the notation.  Note the $H_{q}(\ba,\bbeta;\l;\mathfrak{p})$ function is  written as $H_q\left [\begin{array}{c}{\ba}\\{\bbeta}\end{array};\l\right]$ in \Cref{ss:intro}. We note that $H_q$ does not depend on the order of the parameters in $\ba$ and $\bbeta$ whereas the $\P$-function does, see \cite{Win3X}. In spite of this advantage to using the $H_q$ function, the $\P$-notation more closely relates to periods in the classical setting, and so yields cleaner results in our method.

Another important condition on hypergeometric data is the property of being \emph{defined over $\Q$}. We say that a multi-set $\ba = \{r_{1}, \ldots, r_{n}\}$ is defined over $\Q$ if 
\[
    \prod_{j=1}^{n} (X - e^{2 \pi i r_{j}}) \in \Z[X].
\]
A hypergeometric datum $\{\ba, \bbeta\}$ is defined over $\Q$ if both $\ba$ and $\bbeta$ are.  In this case, the value of $H_q(\ba,\bbeta;\l; \wp)$ is in $\Q$ for $\l\in \Q$. Important work of Beukers, Cohen, and Mellit \cite{BCM} extends the $H_{q}$ function to almost all prime ideals of $\Z$ for $\HD$ defined over $\Q$.

For convenience throughout the paper, we adopt the notation
\begin{equation}\label{eq:shorthand-quotient}
    f\left(\frac{a_{1}^{r_{1}},\cdots, a_{m}^{r_{m}}}{b_{1}^{s_{1}},\cdots,b_{n}^{s_{n}}}\right)\colonequals\frac{f(a_{1})^{r_{1}}\cdots f(a_{m})^{r_{m}}}{f(b_{1})^{s_{1}}\cdots f(b_{n})^{s_{n}}}
\end{equation}
where $f$ could be the gamma function $\G$, the Pochhammer symbol $(\cdot)_n$, the $p$-adic gamma function $\G_p$, or the Gauss sum $\g$. Similarly, we use $f\left({a_{1}^{r_{1}},\cdots, a_{m}^{r_{m}}}\right)$ to denote $f(a_1)^{r_{1}}\cdots f(a_m)^{r_{m}}$.  Moreover, for a fixed multiset $\ba = \left\{r_1, \hdots, r_n\right\}$, we will use $f(\ba)$ to denote $f(r_1, r_2, \hdots, r_n)$. 
Further, we use the notation $$P_M=\{p\mid p\, \, \text{prime}, p\equiv 1\hspace{-3mm}\pmod{M}\}$$
throughout this section.

\subsection{Main results} \label{ss:mr}

The next statement is motivated by the inductive formula (\ref{eq:inductive}) with $z=1$ and $q_{i}=1$ for $2 \leq i \leq n$. We first define 
\begin{equation}\label{eq:gamma-defn}
    \gamma(\HD) \colonequals -1+\sum_{i=1}^n (q_i-r_i).
\end{equation}
\par
\begin{theorem}\label{thm:main}
Let  $n=3$ or $4$. Assume $\ba^\flat=\{r_1,\cdots,r_{n-1}\}$, where $0<r_1
\le\cdots\le r_{n-1}<1$, and $\bbeta^\flat=\{1,\cdots,1\}$ (with multiplicity $n-1$), with $r_n,q_n$ such that  $0<r_n<q_n\le 1$ and $r_2<q_n$.  Let $\HD=\{\{r_n\}\cup \ba^\flat,\{q_n\}\cup \bbeta^\flat\}$ and $M=M(\HD)$. 
Further assume $\gamma(\HD) \leq 1$ and that
\begin{enumerate}
    \item there exists a modular function $t=C_1 q^{}+O(q^2) \in\Z[[q]]$ such that 
    \begin{equation}\label{eq:f}
        f_{\HD}(q)\colonequals C_1^{-r_n}\cdot t(q)^{r_n}(1-t(q))^{q_n-r_n-1}F(\ba^\flat,\bbeta^\flat;t(q))q\frac{dt(q)}{t(q)dq}
    \end{equation}
    is a congruence weight-$n$ holomorphic cusp form satisfying that, for each prime $p \in P_M$, $T_p f_{\HD}=\tilde b_p  f_{\HD}$ for some $\tilde b_p$  in $\Z$  where $T_p$ is the $p$th Hecke operator;
    \item for any prime ideal $\wp$ in $\Z[\zeta_M]$ above $p\in P_M$,  
    $$\prod_{i=2}^{n-1}\iota_\wp(r_i)(-1)\iota_\wp(r_n)(C_1)^{-1}   \cdot \P\left (\HD;1;\wp\right)\in \Z.$$  
\end{enumerate}
Then there exists a normalized Hecke eigenform $f_{\HD}^\sharp$ built from $ f_{\HD}$, not necessarily unique, such that for each $p>29$ in $P_M$, $\tilde b_p=a_p(f_{\HD}^\sharp)$. More explicitly,  
\begin{equation}\label{eq:2.9}
\begin{split}    
 a_p(f_{\HD}^\sharp)=&   { (-1)^{n-1}} \iota_\wp(r_n)(C_1)^{-1}  \cdot \prod_{i=2}^{n-1}\iota_\wp(r_i)(-1)\cdot \P(\HD;1;\wp)-\delta_{\gamma(\HD)=1}\cdot \psi_{\HD}(p) \cdot p\\
 \\
 =&-\iota_\wp(r_n)(C_1)^{-1} {J(\iota_\wp(r_n),\iota_\wp(q_n-r_n))}\cdot H_p(\HD;1;\wp)-\delta_{\gamma(\HD)=1}\cdot \psi_{\HD}(p) \cdot p.
\end{split}
\end{equation} 
Here $\delta_{\gamma(\HD)=1}$ is equal to 1 when ${\gamma(\HD)=1}$ and is 0 otherwise and
\begin{equation}\label{eq:sgn}
   \psi_{\HD}(p) \equiv (-1)^{n-1}\cdot{C_1^{(p-1)r_n}}\G_p\left(\frac{ q_n-r_n}{\ba^\flat}\right) \pmod p. 
\end{equation}
\end{theorem}
In terms of Galois representations, this is equivalent to 
\begin{equation}\label{eq:main-Galois}
   \rho_{f_{\HD}^\sharp}|_{G(M)}\simeq \chi_{\HD}\otimes \rho_{\{\HD;1\}} -\delta_{\gamma(\HD)=1}\cdot \psi_{\HD}\epsilon_\ell|_{G(M)},
\end{equation} where $\epsilon_\ell$ denotes the cyclotomic character and $$\chi_{\HD}(\wp)
=\iota_\wp(r_n)(C_1)^{-1}\cdot \prod_{i=1}^{n-1}\iota_\wp(r_i)(-1).$$

\begin{figure}[ht]
    \begin{center}
    \begin{tikzcd}
         & \boxed{\HD^\flat, (r_n=m/e, q_n)} \arrow[d, "\eqref{eq:f}"] &  \\
         & \omega(t)dt/t \arrow[dl, swap, "t=t(q) \text{ by REAB}"] \arrow[dr, "t=u^e"] &  \\
        f_{\HD}(q)dq \arrow[dd, swap, "T_p" {yshift=-5pt}] &  & \omega(u^e)du^e/u^e \arrow[dl, "p^{th} \hspace{1ex} u-\text{coeff } \eqref{eq:u}" ]\\&\text{Truncated HGS}\arrow[dl, "\text{CFGL \Cref{prop:CFGL-isom}}"] &\\
      \tilde b_p 
          &&  \chi_{HD} \P(\HD; 1) \arrow[ll, " \hspace{1ex} \text{Agree modulo } p^2 \text{ by residue sum}" {yshift=-5pt}] \arrow[ul, swap, "\text{Gross--Koblitz \Cref{thm:gk}}"] 
    \end{tikzcd}
\end{center}
    \caption{A visual outline of the proof of \Cref{thm:main}}
    \label{fig:main idea}
\end{figure}

The key idea is that the right-hand side $f(t) = \omega(t)dt/t$ of \eqref{eq:f} is of the form of the integrand in Euler's integral formula \eqref{eq:inductive}.  We specialize this function two ways.  First, using Ramanujan's theory of elliptic functions to alternative bases and letting $t$ be the modular function $t(q)$ in \Cref{thm:main} yields a modular form $f_{\HD}$.  Then, with $e$ denoting the denominator of $r_n$ in lowest terms, we write $f(t)$ as an expansion in $u = t^e$.  By the Euler integral formula, the coefficients of this expansion will be terminating  hypergeometric functions, see \Cref{ss:FGL}.  The commutative formal group law  
yields a congruence modulo $p$.  The mod $p^2$ supercongruences given below in \Cref{thm:supercongruences}, the Weil--Deligne bounds, and the integrality of the Fourier coefficients then allow us to strengthen this congruence to our desired equality.  
\begin{example}\label{eg:2.3}
  For  $\HD^\flat=\{\{\frac12,\frac12,\frac12\},\{1,1,1\}\}$ and $(r_4,q_4)=(\frac12,1)$, we set $t=-64\frac{\eta(q^2)^{24}}{\eta(q)^{24}}$. 
  Then $ f_{\HD}(q)=f_{8.4.a.a}(q^{1/2})$ as in \eqref{eq:HG-f8.4.a.a}. As this is already a Hecke eigenform after $\tau \mapsto 2\tau$, we have $f_{\HD}^\sharp=f_{8.4.a.a}$. In this case  $\gamma(\HD)=1$, and the normalizing factor in \eqref{eq:2.9} in front of  $H_p$ is $$-\iota_\wp(1/2)(-64)^{-1} J(\iota_\wp(1/2),\iota_\wp(1/2))=-\phi(-64)J(\phi,\phi)=1,$$   
  and $$\psi_{\HD}(p)=\phi(-1) (-1)^3\cdot  \G_p\left(\frac{\frac12}{\frac12,\frac12,\frac12}\right)\overset{\eqref{eq:p-gamma-reflect}}=-\phi(-1)(-1)^{(p+1)/2}=1.$$ Thus, \eqref{eq:2.9} says for any odd prime $p$
  $$a_p(f_{8.4.a.a})=H_p(\HD;1)-p,$$ recovering \eqref{eq:AO}.
\end{example}

We now use \Cref{thm:main} and the same modular function as the previous example, to obtain the following new modularity result for a datum of length four. 

\begin{proposition}\label{prop:6.6}For each prime 
$p \equiv 1 \pmod{4}$
    $$a_p(f_{32.4.a.a})= - {\mathbb P}\left(\bigg\{\frac12,\frac12,\frac12,\frac14\bigg\},\bigg\{1,1,1,\frac34\bigg\};1;\mathfrak p \right)-p.$$
\end{proposition}
\begin{proof}Let  $\HD^\flat=\{\{\frac12,\frac12,\frac12\},\{1,1,1\}\}$, $(r_4,q_4)=(\frac14,\frac34)$ and $t=-64\frac{\eta(q^2)^{24}}{\eta(q)^{24}}$. 
In this case, $f_{\HD}(q)=\frac{\eta(q)^{10}}{\eta(q^2)^2}$ satisfies condition (1) of \Cref{thm:main} and $f_\HD^\sharp$ can be taken as
   $$f_{\HD}^\sharp(q)=f_{32.4.a.a}(q)=\frac{\eta(q^4)^{10}}{\eta(q^8)^2}-8\frac{\eta(q^8)^{10}}{\eta(q^4)^2}.$$ 
   Note that $\P(\HD;1;\wp) \in \Z[i]$ and its complex conjugate is $\P(\overline \HD;1;\wp)$,  where $\overline{HD} = \{ \{\frac{1}{2},\frac{1}{2},\frac{1}{2},\frac{3}{4}\},\{1,1,1,\frac{1}{4}\}\}$. Using Proposition 1 of \cite{LLT2}, we have $\P(\overline \HD;1;\wp)=\P(\HD;1;\wp)$. 
   Thus condition (2) of \Cref{thm:main}  is also satisfied. Here $\psi_{\HD}(p)=1$, similar to the previous example.
\end{proof} 
We remark that the form $\frac{\eta(q^8)^{10}}{\eta(q^4)^2}$ arises from the datum $\ol{\HD}$.  The combinations of $f_{\HD}$ and $f_{\ol{\HD}}$ give Hecke eigenforms $f_{32.4.a.a}$ and  $f_{32.4.a.c}$ which differ by a quadratic character of conductor 4.  One important difference between \Cref{eg:2.3} and \Cref{prop:6.6} is that the hypergeometric datum in \Cref{prop:6.6} is not defined over $\Q$. Condition (1) of \Cref{thm:main} is satisfied for both cases by construction. Now Condition (2) of \Cref{thm:main} is automatically satisfied in \Cref{eg:2.3}, as the datum is defined over $\Q$. However, in \Cref{prop:6.6} a character sum identity, such as Proposition 1 of \cite{LLT2}, implies Condition (2) of \Cref{thm:main}, when $p \equiv 1 \pmod{4}$.

In the proof of \Cref{thm:main}, we rely on a more general  $p$-adic  supercongruence result. 
To apply $p$-adic techniques, we embed $\P(\HD;1;\wp)$ into $\C_p$ through the Teichm\"{u}ller character $\omega_q$, and as we are considering prime ideals above splitting primes we may take $q=p$. Further, under the assumption of $\wp$ being above $p \in P_M$, all values of the characters of $\iota_\wp(r_i)$ and $\iota_\wp(q_i)$ can be embedded into $\Z_p$.  In this case we write the embedding of $H_p(\HD;1;\wp)$ into $\Z_p$ as $H_p(\HD;1;{ \bar \omega_p})$, namely 
\begin{equation}\label{eq:Hp-embedded}
    H_p(\HD;1;{ \bar \omega_p}) = H_p(\HD;1;\wp),  \quad \mbox{where } \iota_\wp\left(\frac 1{p-1}\right)={ \bar \omega_p}.
\end{equation}
Likewise, we use $J_{\bar \omega_p}(a,b)$ for the embedding of $J(\iota_\wp(a),\iota_\wp(b)).$ 
\subsection{Supercongruences}
\subsubsection{Supercongruence background}
A key step in the proof of \Cref{thm:main} involves proving supercongruences to strengthen the relationship between the Fourier coefficients $a_p(f_{\HD}^\sharp)$, truncated hypergeometric functions, and hypergeometric character sums arising from the commutative formal group law.  We do this by proving two constituent congruences.  The first, which we say is of `Gross--Koblitz type', gives a congruence modulo $p$ between the character sum $H_p$ and the corresponding hypergeometric series truncated at $p-1$.  The main tool used to establish this congruence is the Gross--Koblitz formula, which we recall below in \Cref{thm:gk}.  Next, the bridge from the truncated hypergeometric sum to the Fourier coefficients $a_p(f^{\sharp}_{\HD})$ is attained through a formal group isomorphism from $f_{\HD}$ to the truncated hypergeometric series as in \Cref{prop:CFGL-isom} and then  by a `Dwork type' supercongruence, wherein we relate our truncated hypergeometric function to the corresponding Dwork unit root introduced in \cite{Dwork}.  We discuss this Dwork unit root and its relationship to Katz's Galois representations briefly in \Cref{ss:Dwork}. In particular, this root can be realized as the $p$-adic embedding of a particular root of the characteristic polynomial of Frobenius of $\rho_{\left\{\HD;1\right\}}$ over $p$, see for example \Cref{rem:char-sum-as-trace}.  For both congruences, we explicitly compute the error to a supercongruence modulo $p^2$, and give criteria under which this error is guaranteed to vanish modulo $p^2$ in terms of the invariant $\gamma(\HD)$ defined in \eqref{eq:gamma-defn}.  Although \Cref{thm:main} only considers $\HD$ of length $3$ or $4$, the methods to establish supercongruences for these cases can be quickly generalized, and so we prove corresponding supercongruences for a much larger collection of hypergeometric data.
\begin{theorem}\label{thm:supercongruences} 
Let $\HD = \left\{\ba, \bbeta\right\}$ where $\ba=\{r_1,\hdots,r_n\},\bbeta=\{q_1,  \hdots, q_n\}$ with $r_i,q_i\in\Q\cap (0,1]$ satisfying
\begin{equation}   \label{eq:*condition} 
    0 < r_1\le r_2 \leq \cdots \le r_n < 1, \quad 0 < q_1 \leq \cdots \leq q_{n-2} \leq q_{n-1} = q_n = 1, \quad q_{i} > r_{i+2} .
\end{equation}
Let $\omega_p$ be the Teichm\"{u}ller character of the finite field $\F_p$.   When $\gamma(\HD)\le 1$ and $\lambda = 1$, for each prime $p\equiv 1\pmod{M(\HD)}$ greater than $n+1$ which is ordinary---meaning that the truncation $F(\ba, \bbeta; 1)_{p-1}$ of $F(\ba, \bbeta; 1)$ after $p$ terms is not divisible by $p$---we have
\begin{equation}\label{eq:super-combined}
    H_p(\ba, \bbeta; 1; \bar{\omega}_p) - \delta_{\gamma(\HD)=1}\G_p\left(\frac{\bbeta}{\ba}\right) p \equiv F(\ba, \bbeta; 1)_{p-1} \equiv \mu_{\ba, \bbeta, 1, p} \pmod{p^2},
\end{equation} 
where $\mu_{\ba, \bbeta, 1, p}$ is the Dwork unit root defined below in \eqref{eq:Dwork-unit-root-congruence}.
\end{theorem} 
\begin{remark}
    In the proof of this Theorem in \Cref{sec:super}, we will in fact show that the assumption $p > n+1$ can be reduced further to $p > \hat{n}+1$, where $\hat{n}$ is the total number of $q \in \bbeta$ which are not equal to $1$.  Computations suggest that no such hypothesis is necessary, and that the supercongruence holds for all $p \equiv 1 \pmod{M(\HD)}$ in most cases.
\end{remark}
\begin{remark}\label{rem:super-pieces}
    As we noted above, the supercongruence \eqref{eq:super-combined} arises from two supercongruences, the Gross--Koblitz type supercongruence 
    \[
        H_p(\ba, \bbeta; \l; \bar{\omega}_p) - E_{\mathrm{GK}}(\ba, \bbeta; \l^p)p \equiv F (\ba, \bbeta; \l^p)_{p-1} \pmod{p^2},
    \]
    and, using $F_s(\ba, \bbeta; \lambda)$ to denote the truncation of $F(\ba, \bbeta; \lambda)$ at $p^s-1$, the Dwork-type congruence
    \[
        F_{s+1}(\ba, \bbeta; \l) - p \l^p F_s'(\ba, \bbeta; \l^{p}) E_{\mathrm{Dwork}}(\ba, \bbeta; \l) \equiv F_s (\ba, \bbeta; \l^{p}) F_1(\ba, \bbeta; \l) \pmod{p^2},
    \]
    both valid at primes $p \equiv 1 \pmod{M}$.
    The $p$-linear error terms $E_{\mathrm{GK}}$ and $E_{\mathrm{Dwork}}$ are defined explicitly in \eqref{eq:E(GK)} and \eqref{eq:EDwork}, respectively.  We obtain these congruences in \Cref{lem:GK-congruence} and \Cref{lem:Dwork-congruence}, respectively, and in \Cref{prop:GK-error} and \Cref{prop:Dwork-error} compute them explicitly in the case $\gamma(\HD) \leq 1$ to obtain \eqref{eq:super-combined}.  
\end{remark}

Our method can be applied for different choices of $\HD^\flat$.  In a private communication, Frits Beukers and Henri Cohen determined appropriate triangle groups, Hauptmoduln, and special values of classical hypergeometric functions for many more cases of $\HD^\flat$, vastly generalizing our discussions in \Cref{sec:triangle}. Below we apply our method to  one case in which $\HD$ has length three in which $q_3\neq 1$. 

\begin{theorem}\label{thm:length3} Let $ \ba^\flat=\{\frac12,\frac12\},\bbeta^\flat=\{1,1\}$, and $(r_3,q_3)=(\frac14,\frac34)$ so that $\HD=\{\{\frac12,\frac12,\frac14\},\{1,1,\frac34\}\}$.
For any prime ideal $\wp$ above $p\equiv 1\pmod4$,  
\begin{equation}\label{eq:eg1-P}
    {\mathbb P}\left(\bigg\{\frac12,\frac12,\frac14\bigg\},\bigg\{1,1,\frac34\bigg\};1;\wp\right)=
    a_p(f_{32.3.c.a}),
\end{equation}
where $f_{32.3.c.a}$ is as given by \eqref{eq:K2(1/4,3/4)}. Moreover, for each prime $p \equiv 1 \pmod 4$ we have the corresponding supercongruence
\begin{equation}\label{eq:eg1-super} 
     \frac{\Gamma_{p}(\frac{1}{4})\Gamma_{p}(\frac{1}{2})}{\Gamma_{p}(\frac{3}{4})}\pFq32{\,\,\frac12&\frac12&\frac1{4}}{\,\,1&1&\frac3{4}}1_{p-1}\equiv a_p(f_{32.3.c.a}) \pmod{p^2}. 
\end{equation}
\end{theorem}
\subsection{\texorpdfstring{The $p$-adic gamma function and the Gross--Koblitz formula}{The p-adic gamma function and the Gross--Koblitz formula}}
To prove congruences in later sections, we use the $p$-adic perturbation technique introduced in \cite{Long,LR} and further developed in \cite{Allen,LTYZ}.
We will make extensive use of Morita's $p$-adic gamma function, which is defined on positive integers $n$ by
\begin{equation}\label{eq:p-gamma-defn}
    \Gamma_p(n) = (-1)^n \prod_{\substack{1 \leq i \leq n-1 \\ p \nmid i}} i
\end{equation}
and then extended continuously to $\Z_p$.  Like the classical gamma function, this satisfies a shifting property
\begin{equation}\label{eq:p-gamma-shift}
    \frac{\G_p(x+1)}{\G_p(x)} = \begin{cases} -x & \text{if } x \in \Z_p^\times \\ -1 & \text{if } x \in p \Z_p \end{cases}
\end{equation}
and a reflection property
\begin{equation}\label{eq:p-gamma-reflect}
    \G_p(x) \G_p(1-x) = (-1)^{x_0}
\end{equation}
where $x_0 \in \left\{1, 2, \hdots, p\right\}$ satisfies $x \equiv x_0 \pmod{p}$.

For any multiplicative character $A$ of a finite field $\F_q$ of characteristic $p$, we use 
\begin{equation}\label{eq:gauss-defn}
    \g(A)\colonequals\sum_{x\in \F_q^{\times}}A(x)\zeta_p^{\text{Tr}^{\F_q}_{\F_p}(x)}
\end{equation}
to denote the \emph{Gauss sum} of $A$, where $\text{Tr}^{\F_q}_{\F_p}$ is the standard trace map from $\fq$ to $\fp$.
We will also use the Gross--Koblitz formula which relates Gauss sums to $p$-adic gamma functions.
\begin{theorem}[Gross--Koblitz, \cite{Gross-Koblitz}] \label{thm:gk}
    Let $p$ be a prime and $0 \leq r \leq p-2$ be an integer. Then
    \[
        \g(\bar{\omega}_p^{r}) = - \pi_{p}^{r} \cdot \Gamma_{p} \bigg( \frac{r}{p-1} \bigg),
    \]
    where $\omega_p$ is the Teichm\"{u}ller character of $\fp^{\times}$, 
    and $\pi_{p}$ is a fixed root of $x^{p-1} + p = 0$ in $\C_{p}$.
\end{theorem}
This is instrumental in the proof of the first supercongruence in \Cref{rem:super-pieces}, as it allows us to translate from the Gauss sums defining $H_p$ to the gamma functions defining ${}_nF_{n-1}$.  Additionally we recall that, for multiplicative characters $A$ and $B$ of $\F_q^\times$, their Jacobi sum satisfies

\begin{equation} \label{eq:g2j}
J(A,B) = \frac{\g(A)\g(B)}{\g(AB)}, \quad  A\neq \ol B.   
\end{equation}
Combining this with \Cref{thm:gk} gives the following useful corollary.
\begin{corollary}\label{cor:Jac-Gp}
    For $r,s\in \Q\cap (0,1)$, $p\equiv 1\pmod{\text{lcd}(r,s)}$ such that $r+s< 1$.  Then 
    \begin{equation}\label{eq:5.31}
        -J_{\bar \omega_p}(r,s)=
        \Gamma_{p}\bigg(\frac{r,s}{r+s} \bigg).
    \end{equation}
\end{corollary}    

The shorthand notation for Jacobi sums from \Cref{ss:mr} is used above. 

\section{Modular forms and commutative formal group laws}\label{sec:background}
 In this section, we recall some classic results in hypergeometric functions and modular forms that will be used to explain when the function $f_{\HD}$ of \Cref{thm:main} is a modular form.  We then provide a $p$-adic relation between these topics by commutative formal group laws.  For definitions and basic properties of modular forms, $\eta$-products, and Hecke operators, see \cite{Cohen-Stromberg, DS-modularforms}.
\subsection{Ramanujan's theory of elliptic functions to alternative bases (REAB)} \label{subsec: Ramanujan}
For $d \in \{2,3,4,6\}$, let
\[
    \HD_d\colonequals\left\{\left\{1/d,1-1/d\right\},\{1,1\}\right\},
\]
which is a length two hypergeometric datum defined over $\Q$ that corresponds to a second-order hypergeometric differential equation.  The corresponding monodromy group $\G_d$ is isomorphic to the genus zero congruence subgroup $\G(2)$ (which is isomorphic to $\G_0(4)$), $\G_0(3)$, and $\G_0(2)$, when $d=2$, $3$, and $4$, respectively. 
The  Schwarz map (see \cite[Theorem 3.2]{Win3X})
\begin{equation}\label{eqn: alt-basis}
    f(z)\colonequals\tau=\frac{i}{\kappa_d}\cdot \frac{F(\HD_d;1-z)}{F(\HD_d;z)},
\end{equation}
where $\kappa_d=1,\sqrt3,\sqrt 2$ when $d=2,3,4$ respectively, sends the complex upper half plane to a hyperbolic triangle with inner angles $0,0,$ and $(1-\frac2d)\pi$.  
Both $F(\HD_d;1-
z)$ and $F(\HD_d;z)$ satisfy
the same hypergeometric differential equation. 
\begin{center}
    \includegraphics[height=3.5cm]{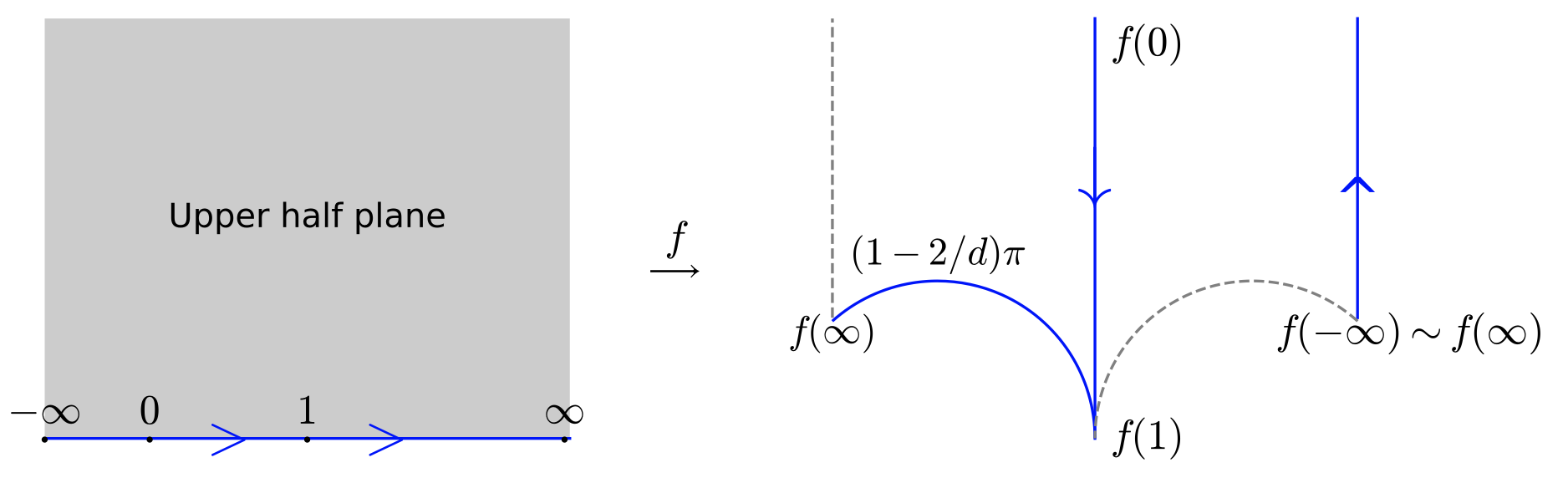}   
\end{center} 
Let $t_d$ be the Hauptmodul of $\G_d$ which takes values 0 and 1 at the two cusps corresponding to $f(0)$ and $f(1)$ and has a simple pole at the elliptic point when $d=3$ and $4$ or at the other cusp when $d=2$.  The map $t_d\mapsto 1-t_d$, arising from a normalizer of $\G_d$ in $\GL_2(\Q)$,  is the involution swapping the two cusps. This fact can be used to relate $L$-values of weight three cusp forms at 1 and 2, details are given in part II \cite{HMM2} of this series.    The Hauptmodul $t_d$ can also be taken as the inverse of the Schwarz map \eqref{eqn: alt-basis}.  In many cases, $t_{d}$ can be expressed as a quotient of Dedekind eta functions.  As an example,
\begin{equation}\label{eq:lambda-in-eta}
    t_2\colonequals\l(\tau)=16 \frac{\eta(\frac{\tau}{2})^{8}\eta(2 \tau)^{16}}{\eta(\tau)^{24}},  
\end{equation} 
the expressions of other $t_d$ are available in Table \ref{tab:Hauptmodul} in  \Cref{sec:triangle}.  Another important fact is that
$\pFq21{\,\,\frac1d&\frac{d-1}d}{\,\,1&1}{t_d(\tau)}$ is a weight one modular form for the group $\G_d$. In the classical theory, the $d=2$ case is well understood.  When $\l(\tau)$ is the modular lambda function, as in \eqref{eq:lambda-in-eta}, one has the following whenever both sides make sense:
\begin{equation}\label{eqn: alt-2}
    \pFq21{\,\,\frac12&\frac12}{\,\,1&1}{\l(\tau)}=\sum_{n,m\in \Z}q^{(n^2+m^2)/2}=\theta_3(\tau)^2.
\end{equation}
Here $\theta_3(\tau)$ is one of the weight-1/2 Jacobi theta functions, see \cite{LLT}. There are similar expressions for $\HD_3$ and $\HD_4$. For $d=6$,  by a hypergeometric quadratic formula,  
\[
   {}_2F_1(\HD_6; t_6)= \pFq21{\,\,\frac1{12}&\frac5{12}}{\,\,1&1}{4t_6(1-t_6)}= \pFq21{\,\,\frac1{12}&\frac5{12}}{\,\,1&1}{\frac{1728}{j}} =E_4^{1/4},
\]
where $E_{4}$ is the weight four normalized Eisenstein series on $\SL_{2}(\Z)$ and $t_6$ is a function satisfying $4t_6(1-t_6)=\frac{1728}{j}$. For other cases, please see the Appendix \Cref{sec:triangle}. 
\begin{example}\label{eg:LC-overC} For $\HD_2$, we consider the Legendre family of elliptic curves
    \begin{equation}\label{eq:LC}
        E_z: \quad y^2=x(1-x)(1-z x).
    \end{equation}
    Note that $ \omega_z\colonequals\frac{dx}{\sqrt{x(1-x)(1-z x)}}$ is the unique up to scalar holomorphic differential 1-form on $E_z$.  Using \eqref{eq:inductive},
    \begin{equation}  
        \int_0^1 \omega_z 
        =\pi \cdot F(\HD_2;z).
    \end{equation}

\end{example}

\subsection{\texorpdfstring{$\BK_2(r,s)$}{K2}-functions}\label{ss:K2}
We now investigate when the differential of the Euler integral formula evaluated at the appropriate Hauptmodul is the modular form $f(q)$ in Theorem \ref{thm:main}. We consider the case when $\HD^\flat=\HD_2$.
\begin{defn}\label{defn:K2}
    Given $r,s\in \Q$, define
    \begin{equation}\label{eq:k2}
        \BK_2(r,s)(\tau)\colonequals \frac{\eta \left(\frac{\tau}{2} \right)^{16s-8r-12}\eta(2 \tau)^{8s+8r-12}}{\eta(\tau)^{24s-30}}.
    \end{equation} 
 \end{defn}
Using \eqref{eqn: alt-2} and $\lambda$ as in \eqref{eq:lambda-in-eta},  we have \begin{equation}\label{eq:alt-2-eta}
        {\l}(\tau)^{r}{(1-\l(\tau))}^{s-r-1}   \pFq21{\,\,\frac 12&\frac 12}{\,\,1&1}{\l(\tau)}  q\frac{d \l(\tau)}{\l(\tau)dq} = 2^{4r -1 } \BK_2(r,s)(\tau),
    \end{equation}
    as stated in \Cref{ss:intro}.  In the following, we will focus on those pairs $(r,s)$ which give rise to congruence modular forms. 

\begin{lemma}\label{lem:E-eta}
 For each $(r,s)\in  {\mathbb S}_2$ in \eqref{eq:S}, $\BK_2(r,s) (N(r)\tau)$ is a congruence weight three holomorphic cusp form of  level $ N(r)N(s-r)$ with Dirichlet character induced by $\left(\frac{-2^{24s}}\cdot\right)$, where  
 \begin{equation}
        N(r)\colonequals\frac {48}{\gcd(24r,24)}.
    \end{equation} 
\end{lemma}
\begin{proof}We first recall from \cite{LLT} that 
\[
  q\frac{d \l(\tau)}{\l(\tau)dq} = \frac 12\theta_4^4(\tau),
\]
where 
\begin{equation}\label{eq:theta4}
    \theta_4(\tau) =\sum_{n\in \Z}(-1)^{n}q^{\frac{n^2}2},  \quad \text{and} \quad \theta_4^4(\tau)=(1-\l(\tau))\theta_3^4(\tau).
\end{equation}
Thus 
\[
    2^{4r-1}\BK_2(r,s)(\tau)=  \left(\frac{\l}{1-\l}\right)^{r}(1-\l)^{s}\theta_3^6(\tau)=q^{r/2}+\cdots, 
\]
which is holomorphic on the upper half-plane. The poles and zeros occur at points on $\Q\cup\{\infty\}$.  Since $\G(2)$ is torsion-free,   $\BK_2(r,s)(\tau)$ is a weakly holomorphic modular form on a subgroup $\G$ of $\G(2)$ determined by the rational numbers $r$ and $s$, and the order of poles or zeroes at rational numbers are determined by the branch cover between the modular curves $X(\G) \rightarrow X(\G(2))$. The orders at $0$, $1$, and $i\infty$ are
\renewcommand{\arraystretch}{1.2}
\[
    \begin{array}{||c||c|c|c||}\hline 

        \mbox{pt }&i\infty& 0& 1\\ \hline 
        \mbox{Ord}_{\mbox{pt }}(\BK_2(r,s)(\tau))& r& \frac 14(s-r)& 3/2-s\\ \hline
    \end{array}
\]
\renewcommand{\arraystretch}{1}
The necessary conditions for being congruence are $24s\in \Z$, $8(r+s)\in \Z$, and $16s-8r\in\Z$. For these $r$ and $s$, $\BK_2(r,s)(48\tau)$ is a cusp form, not necessarily new, of level $48^2$ with Dirichlet character $\left(\frac{-2^{24s}}\cdot\right)$. 
\end{proof}
We now consider conjugate pairs of $(r,s)\in\mathbb S_2$.
\begin{lemma}\label{lem:S2-conjugate}
     For a given $(r,s)\in {\mathbb S}_2$, let $M=\text{lcd}{(\frac12,r,s)}$.  For any $c\in(\Z/M\Z)^\times$, let  
$r_c=\{cr\}$ and $s_c=\{cs\}$  if $\{cr\}\le {\{cs\}}$; and $s_c=\{cs\}+1$ otherwise. Then $(r_c,s_c)\in \mathbb S_2.$
\end{lemma}
\begin{proof}
    From the choices of $(r_c,s_c)$ we can check that $0<r_c<s_c\le \frac 32$. Other conditions can be verified directly.
\end{proof} 
For example, the conjugates of $(1/8,1)$ in $\mathbb S_2$ are $(j/8,1)$ where $j=1,3,5,7$. We use these pairs to  illustrate our method. By equation (\ref{eq:alt-2-eta}), 
\begin{equation}\label{eq:K2(1/8,1)}
       \bigg(\left(\frac{t}{1-t}\right)^{\frac{1}{8}} (1-t) \cdot \pFq{2}{1}{\,\,\frac{1}{2}&\frac{1}{2}}{\,\,1&1}{t} \cdot \frac{d t}{t} \bigg)_{t=\l} = \sqrt{2}  \cdot \frac{\eta(\tau)^{6}\eta \left(\frac{1}{2}\tau \right)^{3}}{\eta(2 \tau)^{3}} \frac{dq^{1/2}}{q^{1/2}}. 
\end{equation}The constant on the right hand side is $C_1^{1/8}=2^{1/2}$, where $C_1=16$ is the leading coefficient of $\l$. In this case, $N(\frac18)=
\frac {48}{3}=16$. Letting $\tau \mapsto 16 \tau$ in the above eta quotient gives
\begin{equation*}
\begin{split}
\BK_2\left(\frac{1}{8}, 1\right)(16\tau) \colonequals \frac{\eta(16 \tau)^{6}\eta(8 \tau)^{3}}{\eta(32 \tau)^{3}} &= q - 3q^{9} -6q^{17} +23q^{25} + 12q^{33} -66q^{41} +\cdots\\
 &= \sum_{{n\equiv 1} {\Mod{8}}} b_nq^n\in S_{3}\left(\Gamma_{0}(256), \left(\frac{-1}{\cdot}\right)\right).
\end{split}
\end{equation*}
Note that for all $j=1,3,5$, and $7$, the forms $\BK_2\left(\frac{j}{8}, 1\right)$ live on the same subgroup of $\G_0(4)$ as newforms.  Hence, we consider the subspace $V$ of $S_{3}\left(\Gamma_{0}(256), \left(\frac{-1}{\cdot}\right)\right)$ spanned by $f_j(\tau):=\BK_2\left(\frac{j}{8}, 1\right)(16\tau)\in q^j(1+\Z[[q^8]])$,  $j=1,3,5, 7$.  One can check that $T_2(f_j)=0$ for all $j$, and the actions of the operators $T_3$, $T_5$, and $T_7$ are as follows:
$$
  \begin{array} {|c||c|c|c|c|}
    \hline
    &f_1&f_3&f_5&f_7      \\\hline 
    T_3& -12f_3&f_1&-4f_7&3f_5  \\\hline
    T_5& 48f_5 &16f_7&f_1&3f_3  \\\hline
    T_7& -64f_7&16f_5&-4f_3&f_1  \\\hline
  \end{array}
$$
From the hypergeometric arithmetic perspective (for detailed discussion, see \cite{ENRosenK2}),  the space $V$ is invariant under the Hecke operators restricted on $V$ and hence give the Hecke eigenforms for this space. These computations lead to the following conclusion.
\begin{corollary}\label{cor:Galois} The space $V$ is an invariant subspace of $S_{3}(\Gamma_{0}(256), \chi_{-1})$ for all Hecke operators. 
The Hecke algebra for the subspace $V$ is generated by $T_3$ and $T_5$, with 
    $$
      T_3^2=-12, \quad T_5^2=48, \quad T_3T_5=3T_7. 
    $$
    Moreover, 
\begin{enumerate}
    \item 
    The minimal polynomials of the Hecke operators have degree at most two. 
    \item  The corresponding newform orbit is 256.3.c.g of LMFDB. 
    Namely,    
    \begin{align*}
         f_{256.3.c.g}(\tau)  &=  \left(\BK_2\left (\frac18,1\right) + a_3  \BK_2\left (\frac38,1\right)+a_5 \BK_2\left (\frac58,1\right)+\frac{a_3a_5}3  \BK_2\left (\frac78,1\right)\right)(16\tau)\\
         &= \eta(16 \tau)^{6} \left(\frac{\eta(8 \tau)^{3}}{\eta(32 \tau)^{3}} + a_3 \frac{\eta(8 \tau)}{\eta(32 \tau)} +a_5 \frac{\eta(32 \tau)}{\eta(8 \tau)} +\frac13a_3a_5 \frac{\eta(32 \tau)^{3}}{\eta(8 \tau)^{3}}\right)
    \end{align*}

where $a_3^2 =-12$, and  $a_5^2=48$.
   
    \item  The $q$-coefficients of $\BK_2\left(\frac{1}{8}, 1\right) (16\tau)$ are multiplicative. In particular, for $p\equiv 1\pmod 8$, they satisfy three-term Hecke recursions. 
\end{enumerate}
\end{corollary}

The above discussion says that for each $j\in\{1,3,5,7\}$, the $\BK_2(\frac j8,1)$ functions are in the same Hecke orbit.  We refer to this situation as the \emph{Galois} case. \begin{defn}\label{defn:galois}
    A pair $(r,s)\in\mathbb S_2$  is said to be in a \emph{Galois orbit} for the $\BK_2$-family if for any $c\in (\Z/N\Z)^\times$ where $N=\mathrm{lcd}(r,s)$, there exists $(r_c,s_c)\in \mathbb S_2$ such that $cr-r_c, cs-s_c\in\Z$,     such that $\mathbb K_2(r,s)$ and $\mathbb K_2(r_c,s_c)$ are in the same Hecke orbit. 
\end{defn} See the second paper of this series \cite{HMM2} and \cite{ENRosenK2} by Rosen for further discussions of the Galois cases and the general construction of the Hecke eigenforms. 

\begin{remark}\label{rem:1/8case}
    The Hecke eigenform $f_{256.3.c.g}$ above depends on a choice of $a_3$ and $a_5$. This demonstrates the non-uniqueness of $f^{\sharp}_{\HD}$ in \Cref{thm:main}.  When we specialize \Cref{thm:main} to \Cref{thm:mainK2version} by setting $\HD^\flat=\HD_2$, condition (1) is equivalent to $(r,s)\in\mathbb S_2$ and (2) is equivalent to $(r,s)$ being in a Galois case. See \cite[Theorem 3.4]{HMM2}.
\end{remark}

\subsection{Commutative Formal Group Law (CFGL)}\label{ss:CFGL}
The goal of the next few subsections is to show \eqref{eq:2.9} agrees modulo $p$ through the Commutative Formal Group Law  (CFGL) isomorphism.
\par
We start with notation. Let $p$ be a prime and $R$ be a $\Z_p$-algebra equipped with an endomorphism $\sigma:R\rightarrow R$ satisfying that $\sigma(a)-a^p\in pR$ for all $a \in R$. For example, let $p$ be a prime which is congruent to 3 or 5 modulo 8 and $R=\Z_p[\sqrt 2]$, then for $x+y\sqrt2\in R$ with $x,y\in\Z_p$ the map $\sigma$ defined by $\sigma(x+y\sqrt2)=x-y\sqrt2$ is such an endomorphism. 
The basic background is as follows. Assume $b_{n} \in R$ for all $n \geq 1$ and consider $$\omega(x) = \sum_{n=1}^{\infty} b_{n}x^{n-1} {dx},\quad \ell(x)=\int \omega(x)=\sum_{n=1}^\infty \frac{b_n}nx^n.$$ If there is another local uniformizer $u$ such that $x(u) = \sum_{n=1}^{\infty} a_{n}u^{n}\in R[[u]]$ and $a_1\in R^\times$, let $\tilde \ell(u)= \int \omega(x(u))$ be a power series in $u$. Then the formal groups $G(s,t)=\ell^{-1}(\ell(s)+\ell(t))$ and $\tilde G(s,t)=\tilde \ell^{-1}(\tilde \ell(s)+\tilde \ell(t))$ are isomorphic, c.f. \cite[A.3]{Stienstra-Beukers}. See \cite[Appendix]{Stienstra-Beukers} for basic examples and more details. Due to the modular form background of our setting, we will use the following version of the CFGL property, which gives $p$-adic analogues of the Hecke recursions satisfied by Hecke eigenforms. 

\begin{proposition}[\cite{Beukers-another},\cite{Stienstra-Beukers}]\label{prop:CFGL-isom}
    Let $p$, $R$ and $\sigma$ be as above. Let $\omega(x) = \sum_{n=1}^{\infty} b_{n}x^{n-1} {dx}$ with $b_{n} \in R$ for all $n \geq 1$. Let $x(u) = \sum_{n=1}^{\infty} a_{n}u^{n}\in R[[u]]$ and suppose $\omega(x(u)) = \sum_{n=1}^{\infty} c_{n}u^{n-1} {du}$ with $c_n\in R$. 

    If there exists $\alpha_{p}, \beta_{p} \in R$ with $\beta_{p}\in pR$ such that for all $m,r \in \N$,
    \begin{equation}\label{eqn:b congruence}
        b_{mp^{r}}-\sigma(\alpha_{p})b_{mp^{r-1}} + \sigma^2(\beta_{p})b_{mp^{r-2}} \equiv 0 \pmod{p^r};
    \end{equation}
    then for all $m,r \in \N$
    \begin{equation}\label{eqn:c congruence}
        c_{mp^{r}} -\sigma(\alpha_{p})c_{mp^{r-1}} + \sigma^2(\beta_{p})c_{mp^{r-2}} \equiv 0 \pmod{p^r}.
    \end{equation}
If $a_{1}$ is invertible in $R$ then (\ref{eqn:c congruence}) implies (\ref{eqn:b congruence}). Moreover, if $b_p/b_1\in R^\times$, then there exists $\mu_p\in  R^\times$ such that for $m,r\ge 1$
    \begin{equation}\label{eq:CFGL-unit}
        b_{mp^r}\equiv \mu_p \sigma(b_{mp^{r-1}})\pmod {p^r},\quad \text{and} \quad  c_{mp^r}\equiv \mu_p \sigma(c_{mp^{r-1}})\pmod {p^r}.
    \end{equation}
\end{proposition} 
Note that if $b_p/b_1\in R^\times$, which is referred to as the ordinary case, then by Hensel's lifting lemma, one root of the left-hand side of \eqref{eqn:b congruence} is $\mu_p$ which also satisfies that  $\mu_p\equiv b_p/b_1\equiv \sigma(\alpha_p)$ modulo $pR$. Also $\mu_p$ is referred to as the \emph{unit root} of $\omega(t)$. 
\par
Below this proposition will be applied to  the differential \eqref{eq:f} using two distinct but related local parameters to represent the function $t$. 
\subsection{The \texorpdfstring{$u$}{u}-coefficients} \label{ss:FGL}
We first express the $t$-expansion of \eqref{eq:f} in terms of terminating hypergeometric series. Throughout, the datum $\{\ba^\flat,\bbeta^\flat\}$ and $r_n,q_n$ are given as in Theorem \ref{thm:main}. For 
$\ba=\{r_1,\cdots,r_n\}$, we use $(\boldsymbol{ \alpha})_k$ to denote $\prod_{i=1}^n (r_i)_k$, as in \eqref{eq:shorthand-quotient}, where $(r)_k$ is the Pochhammer symbol. 
To ease notation, let
\[
    \gamma_n\colonequals-1+q_n-r_n.
\]
\par
\begin{lemma}\label{lem:3.3}
     Let $\{\ba^\flat,\bbeta^\flat\}$ be  as in Theorem \ref{thm:main}, and $\gamma_n$ as above. Then
$$(1-t)^{\gamma_n}F\left(\begin{array}{c}\ba^\flat\\\bbeta^\flat\end{array};t\right)=\sum_{k\ge 0} A_{\gamma_n}(k)\cdot F\left( \begin{array}{c} \{-k\}\cup\ba^\flat \\ \{q_n-r_n-k\}\cup  \bbeta^\flat \end{array};1\right)t^k, $$ where    
     $F\left( \begin{array}{c} \{-k\}\cup\ba^\flat\\ \{q_n-r_n-k\}\cup  \bbeta^\flat\end{array};1\right)$ is a terminating series as $-k\in\Z_{<0}$ and 
\begin{equation}\label{eq:A}
    A_{\gamma_n}(k)\colonequals\frac{(-\gamma_n)_k}{k!}=(-1)^k\binom{\gamma_n}k.  
\end{equation} 
\end{lemma} 
\begin{proof}
By the binomial 
theorem, $(1-t)^{\gamma_n} = \displaystyle\sum_{k_{1} \geq 0} \frac{(-\gamma_n)_{k_{1}}}{k_{1}!}t^{k_{1}}= \displaystyle\sum_{k_{1} \geq 0} A_{\gamma_n}(k_1)t^{k_1}.$
Thus
\begin{eqnarray*}
(1-t)^{\gamma_n}F\left(\begin{array}{c}\ba^\flat\\\bbeta^\flat\end{array};t\right) &= &\sum_{k_{1},k_{2}=0}^{\infty}\frac{(-\gamma_n)_{k_1}}{k_1!}\frac{(\boldsymbol{ \ba}^\flat)_{k_2}}{(\boldsymbol{ \bbeta}^\flat)_{k_2} }t^{k_1+k_2}\\ 
&\overset{k=k_{1}+k_{2}}=& \sum_{k=0}^{\infty} \sum_{k_2=0}^{k} \frac{(-\gamma_n)_{k-k_2}}{(k-k_2)!} \frac{(\boldsymbol{ \alpha}^\flat)_{k_2}}{(\boldsymbol{ \bbeta}^\flat)_{k_2} }t^{k}\\
&=& \sum_{k\ge 0} A_{\gamma_n}(k)\,F\left( \begin{array}{c} \{-k\}\cup\ba^\flat\\ \{q_n-r_n-k\}\cup  \bbeta^\flat\end{array};1\right)t^k.
\end{eqnarray*}
In the last step, we use the identity $(a)_{k-i} = (-1)^{i} \frac{(a)_{k}}{(1-a-k)_{i}}$  from \cite{Win3X}. 
\end{proof}
Assume $r_n=\frac{m}{e}$ is in lowest terms with $e\ge 1$ and let $u=t^{1/e}$. We now multiply the expression in \Cref{lem:3.3} by $t^{r_{n}} dt/t$ to match with \eqref{eq:f}, up to a constant.  This yields the $u$-expansion from the right-hand side of \Cref{fig:main idea}
\begin{equation}\label{eq:u}
  t^{r_{n}}(1-t)^{\gamma_n}F(\ba^\flat, \bbeta^\flat; t) \frac{dt}{t}
   =  e\cdot\sum_{k=0}^{\infty}A_{\gamma_n}(k) \,F\left( \begin{array}{c} \{-k\}\cup\ba^\flat\\ \{q_n-r_n-k\}\cup  \bbeta^\flat\end{array};1\right) u^{ke+m}  \frac{du}{u}.
\end{equation} 
\par
We now take a closer look at the coefficients of the formal power series
$$\sum_{k=0}^{\infty}A_{\gamma_n}(k)\, F\left( \begin{array}{c} \{-k\}\cup\ba^\flat\\ \{q_n-r_n-k\}\cup  \bbeta^\flat\end{array};1\right) u^{ke+m}=u^m+\sum_{n>m}c_n u^m, \quad c_n\in \Q $$ appearing on the right hand side of  \eqref{eq:u}. From the above expression, $c_m=1$. In light of \eqref{eq:CFGL-unit}, we look at $c_{mp}$ where  $p\equiv 1\pmod {M(\HD)}$ and particularly $m$ being the numerator of $r_n$. From equating $ke+m=pm$, we get $$k=k_0 \colonequals (p-1)\frac{m}e=(p-1)r_n\in \Z.$$ 

First, we  relate the additional rescaling term $A_{\gamma_n}(k_0)=A_{\gamma_n}((p-1)r_n)$ to the value of a Jacobi sum modulo $p$ via the Gross--Koblitz formula recalled in \Cref{thm:gk}. 
\begin{lemma} \label{eq:gkp}
Assume the notations as before and that $p \equiv 1 \pmod{M}$. Then
    \begin{equation}\label{eq:Jac-Gp2}
        A_{\gamma_n}((p-1)r_n)\equiv  \Gamma_{p}\bigg(\frac{r_n,q_n-r_n}{q_n} \bigg)  \equiv  -J_{\bar \omega_p}(r_n,q_n-r_n) \pmod p. 
    \end{equation}    
\end{lemma}
\begin{proof} Since $M=M(\HD)$, $p\equiv 1 \pmod M$, $(p-1)r_n\in \Z_{\ge 0}$. Thus
    \begin{eqnarray*}
      A_{\gamma_n}((p-1)r_n)&= &   \frac{(1+r_n-q_n)_{(p-1)r_n}}{((p-1)r_n)!}=\G\left (\frac{1+r_n-q_n+(p-1)r_n,1}{1+r_n-q_n,(p-1)r_n+1} \right)\\ &\overset{\eqref{eq:p-gamma-shift}}=&\G_p\left (\frac{1+r_n-q_n+(p-1)r_n,1}{1+r_n-q_n,(p-1)r_n+1} \right)\\&\equiv& \G_p\left (\frac{1-q_n,1}{1+r_n-q_n,1-r_n} \right) \stackrel{\eqref{eq:p-gamma-reflect}}{\equiv} \G_p\left (\frac{q_n-r_n,r_n}{q_n} \right) \pmod p\\
      &\overset{\eqref{eq:5.31}}\equiv&  -J_{\bar \omega_p}(r_n,q_n-r_n) \pmod p.
    \end{eqnarray*}
\end{proof}

Next we consider the terminating series 
$ F\left( \begin{array}{c} \{-k_0\}\cup\ba^\flat\\ \{q_n-r_n-k_0\}\cup  \bbeta^\flat\end{array};1\right) $. 
 For indices $0\le i\le (p-1)r_n+1$, the assumption $1\ge q_n>r_n>0$ of \Cref{thm:main} implies $(q_n-r_n-k_0)_i=(q_n-r_n+(1-p)r_n)_i\in\Z_p$. Further for $i$ in this range,  $$((1-p)r_n)_i\equiv (r_n)_i\pmod p,\quad (q_n-r_n+(1-p)r_n)_i\equiv (q_n)_i\pmod p.$$ Together with the additional assumption of $q_n$ being larger than at least two of the $r_i$ (see \Cref{rem:lowest-weight}), one has
\begin{lemma}Under the assumptions of \Cref{thm:main}, let $p\equiv 1\pmod{M(\HD)}$ be a prime. Then
    
\begin{multline}
    F\left( \begin{array}{c} \{(1-p)r_n\}\cup\ba^\flat\\ \{q_n-r_n+(1-p)r_n\}\cup  \bbeta^\flat\end{array};1\right)\\ \equiv F\left( \begin{array}{c} \{r_n\}\cup\ba^\flat\\ \{q_n\}\cup  \bbeta^\flat\end{array};1\right)_{p-1} 
    =F\left(  \ba, \bbeta;1\right)_{p-1} \pmod {p}.
\end{multline}
\end{lemma}
Altogether, we have the following description for $c_{mp}$:
\begin{equation}\label{eq:cmp}
    c_{mp}\equiv  -J_{\bar \omega_p}(r_n,q_n-r_n)F\left(  \ba, \bbeta;1\right)_{p-1}  \pmod {p}.
\end{equation}
\subsection{The \texorpdfstring{$q$}{q}-coefficients and Hecke recursions}  In this section, we express the $t$-expansion of \eqref{eq:f} in terms of Fourier $q$-series.  From \Cref{subsec: Ramanujan}, we see that if we choose $t$ as an appropriate modular function of the form 
\[
    t =C_1q + O(q^2) \in  \Z[[q]],\quad C_1\neq 0
\] then the left-hand side of  \Cref{eq:u} will become a holomorphic modular form with desired properties, see for example \Cref{lem:E-eta}. We now turn our attention to the $q_e=q^{1/e}$-coefficients of this modular form, where $e$ is the denominator of $r_n$ as before.  Note that both $(1-t(q))^{\gamma_n}F(\ba^{\flat}, \bbeta^{\flat};t) q{dt(q)}/(t(q)dq)$ and $ t(q)^{r_{n}}/(C_1q)^{r_n}$ are locally formal power series in $q$. Thus as in \eqref{eq:f}
\begin{equation}\label{eq:B-coeff}
   f_{\HD}dq=  C_1^{-r_n}t^{r_{n}}(1-t)^{\gamma_n}F(\ba^{\flat}, \bbeta^{\flat};t) dt/t
    =e\sum_{l \geq 1} b_{l}q_e^{(l-1)}{d q_e},    
\end{equation}
where $b_l\in\Q$. From the assumption $T_p f_{\HD} = \tilde{b}_pf_{\HD}$ in Theorem \ref{thm:main}, the Hecke recursions imply that for every prime $p\equiv 1\pmod {M}$ where $M=M(\HD)$ and integers $l,r\ge 1$
\begin{equation}\label{eq:Hecke}
     b_{lp^r} - \tilde b_p\cdot b_{lp^{r-1}} + \varphi(p)p^{k-1} \cdot b_{lp^{r-2}} \equiv 0 \pmod{p^{r}} 
\end{equation} 
where $\varphi$ is the character of $f_{\HD}$. Note that $k\ge 3$ is the weight of $f_{\HD}$. When $r=1$, this reduces to 
\begin{equation}\label{eq:r=1}
    b_{lp} \equiv  \tilde b_p b_l \pmod p.
\end{equation} 
Further in the ordinary case, the unit root as \eqref{eq:CFGL-unit} can be computed by 
\begin{equation}\label{eq:u_f}
   u_{f_{\HD},p}\equiv b_{lp^r}/b_{lp^{r-1}} \pmod{p^r}. 
\end{equation}
\subsection{CFGLs isomorphism in this setting}
We now apply \Cref{prop:CFGL-isom} to relate the $u$ and $q$-coefficients in the previous subsections. Namely 
\begin{equation}\label{eq:qe-expansion}
\begin{split}
   \omega(t)=C_1^{-r_n}t^{r_{n}-1}(1-t)^{-(r_{n}+1-q_{n})}F(\ba^{\flat}, \bbeta^{\flat};t) {dt}/e\overset{\eqref{eq:u}}=& C_1^{-r_n}\sum _{l\ge 1}c_l u^{l-1} du\\
   \overset{\eqref{eq:B-coeff}}=&\sum_{l\ge 1} b_lq_e^{(l-1)}{d q_e}    
\end{split}
\end{equation} 
where $u=t^{1/e}$, $c_l\in \Q$ as given in Lemma \ref{lem:3.3} and in particular $c_m=1$.  Note that $M$ is a multiple of $e$ by definition. Let $p \equiv 1 \pmod{M}$ be a prime, $$R= \Z_p[C_1^{-r_n}], 
\quad \sigma(C_1^{-r_n})=C_1^{-pr_n}.$$ Thus if $C_1\in\Z_p^\times$, then 
\begin{equation}\label{eq:C1-quotient}
   C_1^{-r_n}/\sigma(C_1^{-r_n})\equiv C_1^{(p-1)r_n}\equiv \iota_p(r_n)(C_1)\pmod {pR}, 
\end{equation}  where $ \iota_p({r_n})\colonequals\left (\frac \cdot p\right)^m_e$ is the residue symbol defined in \eqref{eqn:residue symbol}. By Proposition \ref{prop:CFGL-isom}, for $l,r\ge 1$,
\[
    C_1^{-r_n}c_{lp^r}-\tilde b_p\cdot{\sigma(C_1^{-r_n})} c_{lp^{r-1}}+ \varphi(p)p^{k-1} \cdot {\sigma^2(C_1^{-r_n})}c_{lp^{r-2}}\equiv 0\pmod{p^rR}.
\]
\par
When $r=1$ this reduces to $C_1^{-r_n}c_{lp}\equiv\tilde b_p\cdot{\sigma(C_1^{-r_n})} c_{l} \pmod{pR}$. If $C_1\in\Z_p^\times$, using \eqref{eq:C1-quotient}, 
\begin{equation}\label{eq:3.19}
    \iota_p({r_n}) (C_1)\cdot  c_{lp} \equiv   \tilde b_{p}c_{l}\pmod {pR}.
\end{equation} 
In particular, when $l=m$ (the numerator of $r_n$) then using \eqref{eq:cmp} we reach the following conclusion from \eqref{eq:3.19}.
\begin{proposition}\label{prop:CFGL} 
Assume $\ba$ and $\bbeta$ are defined as in \Cref{thm:main}, $0<r_n<q_n<1$ with $r_n = m/e$, $p \equiv 1 \pmod{M}$ is a prime, and $C_1\in\Z_p^\times$. Then  
\begin{equation}\label{eq:cong-mod-p}
    -\iota_p({r_n}) (C_1)\cdot  J_{\bar \omega_p}(r_n,q_n-r_n) \cdot  F(\ba, \bbeta;1)_{p-1}\equiv  \tilde b_{p} \pmod p.
\end{equation}  
\end{proposition}
The  result, when combined with \Cref{lem:GK-congruence} below, is equivalent to the mod $p$ version of \eqref{eq:super-combined}. See the computation below.
\begin{align*}
   -\iota_p({r_n}) (C_1)\cdot & J_{\bar \omega_p}(r_n,q_n-r_n) \cdot \left( H_p({\ba}, {\bbeta}; 1; \bar \omega_p) - \delta_{\gamma(\HD)=1}\G_p\left(\frac{\bbeta}{\ba}\right) p\right) \\
   \equiv&  -\iota_p({r_n}) (C_1)\cdot  J_{\bar \omega_p}(r_n,q_n-r_n) \cdot F(\ba, \bbeta; 1)_{p-1} \\
   \equiv & \, \tilde b_{p} =a_p(f^\sharp_{HD})   \pmod{p}.
\end{align*} 
An explicit application of the results for a datum of length three is given below.
\begin{proposition}\label{prop:K(j/8)}
   Let $\HD=\HD_2$, $r_3=\frac j8,$ where $j=1,3,5$ or 7, and $q_3=1$. Then for each prime $p\equiv 1\pmod 8$,  \begin{equation}\label{eq:3.21}
       \left(\frac2p\right)(-1)^{(p-1)/8} \pFq32{\, \,\frac{j}8&\frac12&\frac12}{\, \,1&1&1}1_{p-1}\equiv  a_p(f_{256.3.c.g}) \pmod p.
   \end{equation}
  
\end{proposition}
\begin{proof}
    Note that   $\gamma_3=-\frac j8.$ Thus 
    \[
        t^{ \frac j8} (1-t)^{-\frac j8}\pFq21{\, \,\frac12&\frac12}{\,\,1&1}{t}\frac{dt}{t} =8\sum_{k\ge 1} \frac{(\frac j8)_k}{k!} \pFq32{\, \,-k&\frac12&\frac12}{\, \, 1&\frac {8-j}8-k&1}1 u^{ 8k+j} \frac{du}u.
    \]
    Letting $t=\lambda$ and noting that  $C_1=16$, we get  \eqref{eq:K2(1/8,1)}. From the discussion in \Cref{ss:K2}, the $q$-expansion (after re-scaling) is a Hecke eigenform for each $T_p$ when $p\equiv 1\pmod 8$. Thus $\tilde b_p=a_p(f_{256.3.c.g})$.  Further,  $\iota_p(\frac j8)(C_1)=\left(\frac{16}p\right)^j_8=\left(\frac2p\right)$ and 
    \[
        A_{\gamma_3}((p-1)r_3)\equiv (-1)^{(p-1)/8}\pmod p.
    \]
    The claim then follows from the \Cref{prop:K(j/8)}.
\end{proof}

In \Cref{sec:super}, the mod $p^2$ supercongruence corresponding to \Cref{prop:CFGL}  will be obtained, as stated in \Cref{thm:supercongruences}. Before that, we recall some basic information about hypergeometric character sums and representations.
%

\section{Hypergeometric  Galois representations }\label{sec:character sums}
\subsection{Katz's theorem} 
We now recall an alternative way from \eqref{eq:P-H} to express the normalized character sum $H_q(\ba,\bbeta;\l;\wp)$ following McCarthy \cite{McCarthy}, which will be used in the next section.
\begin{equation}\label{eq:H}
    H_q(\ba,\bbeta;\l;\wp)\colonequals\frac {1 }{1-q} \sum_{\chi\in \widehat{\kappa_\wp^\times}} \chi((-1)^n\l)
    \prod_{j=1}^n \frac{\g(\iota_\wp(r_j)\chi)}{\g(\iota_\wp(r_j))}
    \frac{\g(\iota_\wp(-q_j) \overline \chi)}{\g(\iota_\wp(-q_j))}, \quad \l \in \kappa_\wp^\times.
\end{equation} 
The condition of $\wp$ being a prime ideal of $\Z[\zeta_{M}]$ can be replaced by being a prime ideal of $\Z[1/M,1/\l]$ when $\HD$ is defined over $\Q$ by work of Beukers, Cohen, and Mellit \cite{BCM}.  More specifically, the importance of $\HD$ being defined over $\Q$ is that at each fiber defined over $\Q$ the corresponding Galois representation, a priori for the Galois group $\text{Gal}(\overline \Q/\Q(\zeta_M))$, can be extended to a representation of the absolute Galois group $G_\Q\colonequals\text{Gal}(\overline \Q/\Q)$ whose trace function is, up to a linear character, given by the $H_{p}$ function.  
\par
In this paper, we relax the condition of $\HD$ being defined over $\Q$.  In the following, we will first recall Katz's result on the hypergeometric Galois representations and then discuss extendable Galois representations. 

Let $G(M)\colonequals\text{Gal}(\overline \Q/\Q(\zeta_M))$, $\iota_\wp$ as in \eqref{eqn:residue symbol} and $\P(\ba,\bbeta;\cdot;\wp)$ \eqref{eq:P-by-induction} as before, so the order of entries in $\ba,\bbeta$ matters.

\begin{theorem}[Katz \cite{Katz90, Katz09}]\label{thm:Katz}  Let  $\ell$ be a prime. Given a primitive pair of multi-sets $\ba=\{r_1,\cdots,r_n\}$, $\bbeta=\{q_1=1,q_2,\cdots,q_n\}$ with $M = M(\HD)$, 
for
any $\l \in \Z[\zeta_M,1/M]\smallsetminus \{0\}$ the following hold. 
\begin{itemize}
\item [i).]There exists an $\ell$-adic Galois representation $\rho_{\{\HD;\l\}}: G(M)\rightarrow GL(W_{\l})$ unramified almost everywhere such that at each nonzero
prime ideal $\mathfrak{p}$ of  $\Z[\zeta_M,1/(M\ell \l)]$ of norm $N(\mathfrak{p})=|\Z[\zeta_M)]/\mathfrak{p}|$ 
\begin{equation}\label{eq:Tr1} \Tr \rho_{\{\HD;\l\}}(\text{Frob}_\mathfrak{p})= (-1)^{n-1}  \iota_{\mathfrak{p}}(r_1) (-1)\cdot \P(\ba,\bbeta; 1/\l;\mathfrak{p}),  
\end{equation} 
where $\Frob_\mathfrak{p}$ denotes the  geometric  Frobenius conjugacy class
of $G(M)$ at $\mathfrak{p}$. 
\item[ii).] When $\l\neq 1$,  the dimension $d \colonequals dim_{\overline \Q_\ell}W_{\l}$ equals $n$ and all roots of the characteristic polynomial of $\rho_{\{\HD;\l\}}(\Frob_\mathfrak{p})$  are algebraic numbers and have the same absolute value $N(\mathfrak{p})^{(n-1)/2}$ under all archimedean embeddings. If $\rho_{\{\HD;\l\}}$ is self-dual, namely isomorphic to any of its complex conjugates, then $W_{\l}$ admits a non-degenerate alternating (resp. symmetric)   bilinear pairing if
$n$ is even and $\gamma(\HD)\in\Z$ (resp. otherwise).
\item[iii).] When $\l=1$, in the self-dual case the dimension is $n-1$. All roots of the Frobenius eigenvalues at $\wp$  have absolute value less or equal to $N(\mathfrak{p})^{(n-1)/2}$.  If $\rho_{\{\HD;\l\}}$ is self-dual, it contains a subrepresentation that admits a non-degenerate alternating (resp. symmetric)   bilinear pairing if
$n$ is even and $\gamma(\HD)\in\Z$ (resp. otherwise). For this subrepresentation, the roots of the characteristic polynomial of $\mathrm{Frob}_{\mathfrak{p}}$ have absolute value exactly $N(\mathfrak{p})^{(n-1)/2}$.
\end{itemize}
\end{theorem}
\begin{example}\label{eg:GLC-Hp}
Let $\ba=\{\frac12,\frac12\},\bbeta=\{1,1\}$. Then for any $\l\in\Q$, with $\l\neq 0,1$, the representation $\rho_{\{{\HD;\l\}}}$ associated to $\HD = \{\ba,\bbeta\}$ is isomorphic to the 2-dimensional $\ell$-adic representation of $G_\Q$ arising from the Legendre elliptic curve $ E_{1/\l}$ given in \eqref{eq:LC} twisted by the Dirichlet Character $\left(\frac{-1}{\cdot}\right)$. 
 Namely for each odd prime $p$ of $\Z$ such that $\l$ can be embedded in $\Z_p^\times$,
\begin{align*}
 \Tr \rho_{\{\HD_2;\l\}}(\text{Frob}_p) &= -\left(\frac{-1}p\right)\P(\ba,\bbeta;{ 1/\l};p) \\ &=H_p(\ba,\bbeta;1/\l;p){= p+1- \left(\frac{-1}p\right) \# (E_{1/\l}/\F_p). }
\end{align*}
\end{example}

\begin{remark}\label{rem:char-sum-as-trace}
    The above theorem implies  $(-1)^{n-1}\P(\ba,\bbeta;\l;\mathfrak{p})$, or its normalization $H_{q}(\ba,\bbeta;\l;\mathfrak{p})$, can be thought of as the trace of a finite-dimensional representation of $G(M)$ evaluated 
    at $\Frob_\mathfrak{p}$.  Hence, each finite hypergeometric function is the sum of the roots of the corresponding characteristic polynomial. 
    For example, when $\HD=\{\{\frac12,\frac12,\frac12,\frac12\},\{1,1,1,1\}\}$, \Cref{eg:2.3}  says $$\rho_{\{\HD;1\}}\simeq \rho_{f_{8.4.a.a}}\oplus \epsilon_\ell.$$  Equivalently, for each odd prime $p$, 
    $$
       \Tr \rho_{\{\HD;1\}}(\text{Frob}_p)=-\left(\frac{-1}p\right)\P(HD;1)=H_p(\HD;1)=a_p(f_{8.4.a.a})+p. 
    $$ 
\end{remark}\smallskip
\subsection{Extendable Galois representations}\label{ss:extendable}
Note that  condition (2) of \Cref{thm:main} is equivalent to $\chi_{\HD}\otimes \rho_{\{\HD;1\}}-\delta_{\gamma(\HD)=1}\psi_{\HD}\epsilon_\ell|_{G(M)}$ of $G(M)$ being extendable to $G_\Q$, see \eqref{eq:main-Galois}. From the representation point of view, this boils down to the following well-known result:
\begin{proposition}\label{prop:extension}Let $M$ be a positive integer.
Assume $\rho$ is a semi-simple finite dimensional $\ell$-adic representation of $G(M)$  which is isomorphic to $\rho^\tau$ for each $\tau\in G_\Q$, then $\rho$ is extendable to $G_\Q$.  Equivalently, for each nonzero prime ideal $\mathfrak{p}$ of $\Z[\zeta_M]$ unramified for $\rho$, $\Tr \,\rho(\Frob_\wp)\in\Z$.  
\end{proposition}
As the Frobenius traces are given by character sums, there are cases the extendability can be seen from character sum identities of \cite{Win3X,Greene,LLT2}. For example, when $\HD=\{\{\frac12,\frac12,\frac14\},\{1,1,\frac34\}\}$ we use Proposition 1 of \cite{LLT2} to obtain 
\begin{equation}\label{eq:Ptransform}
\P\left(\left\{\frac12,\frac12,\frac14\right\},\left\{1,1,\frac34\right\};1;\wp\right)=\P\left(\left\{\frac12,\frac12,\frac34\right\},\left\{1,1,\frac14\right\};1;\wp\right)
\end{equation} for each nonzero prime ideal $\wp\in \Z[\sqrt{-1}]$. 
\par
We now verify the conditions of \Cref{thm:main} for the above example. Note that $f_{\HD}$ in condition (1) can be chosen to be either $\mathbb{K}_{2}\left(\frac{1}{4},\frac{3}{4}\right)(8 \tau)$ or $\mathbb{K}_{2}\left(\frac{3}{4},\frac{5}{4}\right)(8 \tau)$, as these modular forms are in the same Hecke orbit. Further, the identity \eqref{eq:Ptransform} establishes condition (2). Therefore, we conclude that $$\P\left(\left\{\frac12,\frac12,\frac14\right\},\left\{1,1,\frac34\right\};1;p\right) = a_{p}(f_{32.3.c.a}),$$ for primes $p \equiv 1 \pmod{4}$ by \Cref{thm:main} and \eqref{eq:K2(1/4,3/4)}.

Note that condition (2) of \Cref{thm:main} cannot always be established directly with a character sum identity, such as in \eqref{eq:Ptransform}. For example,  to our knowledge, the two functions $\P(\{1/2,1/2,1/8\},\{1,1,1\};1;\mathfrak{p})$ and $\P(\{1/2,1/2,3/8\},\{1,1,1\};1;\mathfrak{p})$ cannot be directly related by character sum identities. There are several other cases where the supplies of character sum identities are seemingly insufficient for condition (2) of \Cref{thm:main}. For these cases, we use a complementary approach via the notation of being Galois, defined in \Cref{defn:galois}.

\par

Applying this Proposition together with part iii) of \Cref{thm:Katz}, we have the following.
\begin{proposition}\label{prop:4.3}
        Let $\HD$ be a length $n \in \{3,4\}$ primitive hypergeometric datum and assume that $\chi_{\HD}\otimes\rho_{\{\HD;1\}}$ of $G(M)$ has an extension $\hat \rho_{\{\HD;1\}}$ to $G_\Q$, which is not unique in general. 
        Then
    \begin{itemize}
        \item[1).] If $n=3$, then  $\hat \rho_{\{\HD;1\}}$ is a 2-dimensional representation of $G_\Q$, at each  prime $p\nmid 2M$, and $\hat \rho_{\{\HD;1\}}(\Frob_p)$ has two roots of the same absolute value $p$. In this case, we let $\hat \rho_{\{\HD;1\}}^{\text{prim}}$ be $\hat \rho_{\{\HD;1\}}$. 
        \item[2).] If $n=4$, then  $\hat \rho_{\{\HD;1\}}$ is a 3-dimensional representation of $G_\Q$. 
        
        2a). If $\gamma(\HD)\in \Z$, then $\hat \rho_{\{\HD;1\}}$ is reducible and has a 2-dimensional subrepresentation $\hat \rho_{\{\HD;1\}}^{\text{prim}}$ of $G_\Q$ whose corresponding eigenvalues at unramified Frobenius element $\Frob_p$  are both of absolute value $p^{3/2}$.
        
        2b). If $\gamma(\HD)\notin \Z$, then $\hat \rho_{\{\HD;1\}}$ could be irreducible.
    \end{itemize}
\end{proposition}
\begin{remark} \label{remakr:length4} 
If $\HD$ has length 4, then assumption (1) of \Cref{thm:main} implies that  $\gamma(\HD) \in \Z$ and hence $\rho_{\{\HD; 1\}}$ falls into the case of 2a) above. To be more precise, 
$f_{HD}$ being a modular form means that $F(\ba^\flat,\bbeta^\flat;t)$ is a modular form on an arithmetic triangle group. The only cases are $\ba^\flat=\{\frac12, \frac1d, 1-\frac1d\}$ for the triangle groups of the form $(2, m, \infty)$, where $m=\infty,3,4$ when $d=2,3,4$ respectively, see \Cref{sec:triangle}. Thus $t$ takes value $1$ at the elliptic point of order $2$ so $\gamma(\HD^\flat)$ is a half integer.  In particular, when $0<r_n<q_n\le 1$ and $q_n-r_n=1/2$,  $\gamma(\HD) \in \Z$.

\end{remark}

In both cases, the conclusion of \Cref{thm:main} can be rephrased as follows.
\begin{proposition}\label{prop:4.8}
Assume the notation and assumptions of \Cref{thm:main}. The representation  $\hat \rho_{\{\HD;1\}}^{\text{prim}}$ is modular. For each prime $p\equiv 1 \pmod{2M}$, 
 $$\Tr \hat \rho_{\{\HD;1\}}^{\text{prim}}(\Frob_p)=\tilde b_p, \, \text{ where } \,  T_pf_{\HD}=\tilde b_pf_{\HD}.$$ 
\end{proposition}
As was the case in the statement of \Cref{thm:main}, we could consider all primes equivalent to $1$ modulo $M$ with the addition of a sign.  The choice in the above proposition to restrict to $p \equiv 1 \pmod{2M}$ is only made to remove this sign for a more succinct statement.

\section{Supercongruences}\label{sec:super}
\subsection{\texorpdfstring{$p$}{p}-adic background}
This section is dedicated to the proof of \Cref{thm:supercongruences}.  Here we will make ample use of the function $\G_p$ defined in \eqref{eq:p-gamma-defn}.  
Given a multiset $\ba = \left\{r_1, r_2, \hdots, r_n\right\}$ we take the convention
\begin{equation}\label{eq:multiset-addition}
    \boldsymbol{\ba}+a = \left\{r_1+a,r_2+a,\hdots,r_n+a\right\}.
\end{equation}
Throughout the remainder of this section, we will combine the shorthand notations from \eqref{eq:shorthand-quotient} and \eqref{eq:multiset-addition}.  For example, in \eqref{eq:poch-quotient-to-p-gamma} below,
\[
    \G_p\left(\frac{\bbeta+k_0,\ba+k_0+k_1p}{\ba+k_0,\bbeta+k_0+k_1p}\right) = \prod_{j=1}^n \frac{\G_p(q_i+k_0)\G_p(r_i+k_0+k_1p)}{\G_p(r_i+k_0)\G_p(q_i+k_0+k_1p)}.
\]
Additionally, without loss of generality, we assume going forward that $\ba, \bbeta \subset \Q \cap (0,1]$ are ordered so that
\begin{equation*}\label{multi-ordering}
    0 < r_1 \le r_2 \le \hdots \le r_n < 1, \qquad 0 < q_1 \leq q_2 \leq \hdots \leq q_n \leq 1
\end{equation*}
Given some $a \in \Z_p$, we use $0 \leq [a]_0 \leq p-1$ to denote the first $p$-adic digit of $a$.  
Given a hypergeometric datum $\HD = \left\{\ba, \bbeta\right\}$, for each $r_i \in \ba$, $q_i \in \bbeta$ we set
\begin{equation}\label{eq:ai-bi-defn}
    a_i \colonequals [-r_i]_0, \qquad b_i \colonequals [-q_i]_0.
\end{equation}
In particular, every $r_i+a_i$ and $b_i+q_i$ is divisible by $p$, and so the $p$-adic valuation of $\frac{(\ba)_k}{(\bbeta)_k}$ increases by 1 at each $a_i+1$ and decreases at each $b_i+1$.
\begin{lemma}\label{lem:dwork-fixed}
    Let $\HD = \left\{\ba, \bbeta\right\}$, where each element of $\ba$ and $\bbeta$ belongs to $(0,1]\cap\Q$.  Set $M = \mathrm{lcd}(\left\{\ba, \bbeta\right\})$ and let $p \equiv 1 \pmod{M}$.  Then the $p$-adic Dwork dash operation 
    \begin{equation}\label{eq:dwork-dash-defn}
        \cdot ': \Z_p \to \Z_p \qquad r' = \frac{r+[-r]_0}{p}
    \end{equation}
    fixes every element of $\ba$ and $\bbeta$.  Additionally, for each $r_i \in \ba$ and $q_i \in \bbeta$ we have
    \begin{equation}\label{eq:ai-bi-explicit}
        a_i = r_i(p-1) \qquad b_i = q_i(p-1),
    \end{equation} 
    where $a_i$ and $b_i$ are defined as in \eqref{eq:ai-bi-defn}.
\end{lemma}
\begin{proof}
From the assumptions we know $r_i(p-1)\in\Z\cap [1,p-1]$.  As $r_i+r_i(p-1) = pr_i$ is divisible by $p$, it then follows that $[-r_i]_0=r_i(p-1)$. Thus $r_i'=(r_i+r_i(p-1))/p=r_i$. The argument for the $q_i$'s is the same.
\end{proof}
For any such prime, expanding our index $k$ as $k_0+k_1p$ with $0 \leq k_0 \leq p-1$ we have the following specialization of Corollary 4.2 in \cite{Allen}:
\begin{equation}\label{eq:poch-quotient-to-p-gamma}
    \frac{(\ba)_{k_0+k_1p}}{(\bbeta)_{k_0+k_1p}} = \frac{(\ba)_{k_0}}{(\bbeta)_{k_0}}\frac{(\ba)_{k_1}}{(\bbeta)_{k_1}}  \Lambda_{\ba, \bbeta}(k_0+k_1p) \G_p\left(\frac{\bbeta+k_0,\ba+k_0+k_1p}{\ba+k_0,\bbeta+k_0+k_1p}\right)
\end{equation}
where, with $\nu(k_0,x)$ defined by
\begin{equation}\label{eq:nu-defn}
    \nu(k_0,x) = -\biggr\lfloor \frac{x-k_0}{p-1} \biggr\rfloor = \begin{cases} 0 & \text{if } k_0 \leq x, \\ 1 & \text{if } x < k_0 < p, \end{cases}
\end{equation}
we set
\begin{equation*}\label{eq:Lambda-defn}
    \Lambda_{\ba, \bbeta}(k_0+k_1p) \colonequals \prod_{j=1}^n \left(1+\frac{k_1}{r_j}\right)^{\nu(k_0,a_j)}\left(1+\frac{k_1}{q_j}\right)^{-\nu(k_0,b_j)}. 
\end{equation*}
This $\Lambda$ term corresponds to the discrepancy between the functional equations of $\G$ and $\G_p$.
\begin{remark}\label{rem:lowest-weight}
    The hypotheses in \Cref{thm:supercongruences} that $q_i > r_{i+2}$ for all $i$ and that at least two elements of $\bbeta$ are $1$ are equivalent to assuming that 
    the $p$-adic valuations of the hypergeometric coefficients $\hypcoeff{k}$ with $a_2+1 \leq k \leq p-1$ will be divisible by $p^2$, as is illustrated in \Cref{fig:p-slope}.
\end{remark}
\begin{figure}[ht]
    \centering    \includegraphics[height=5cm]{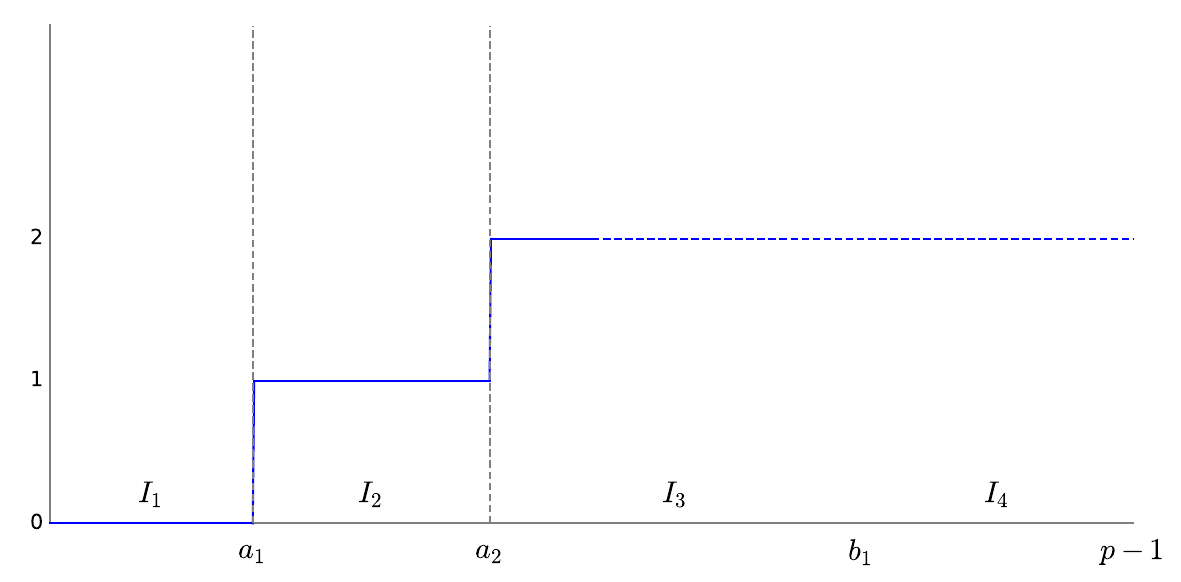}
    \caption{A lower bound on the $p$-adic valuation of $\frac{(\ba)_k}{(\bbeta)_k}$ for $k$ ranging from $0$ to $p-1$ and for the hypergeometric data appearing in \Cref{thm:supercongruences}.  The valuations on the intervals labeled $I_1$ and $I_2$ are exactly $0$ and $1$, whereas on the intervals $I_3$ and $I_4$ we only know that the valuations are greater than or equal to $2$.}
    \label{fig:p-slope}
\end{figure}

As we aim to establish congruence modulo $p^2$ we will only need to consider this $\Lambda$ for $k_0 \leq a_2$.  On this interval, we have
\begin{equation}\label{eq:Lambda-simple}
    \Lambda_{\ba, \bbeta}(k_0+k_1p) = 
    \begin{cases}
        1 & 0 \leq k_0 \leq a_1; \\
        1+\frac{k_1}{r_1} & a_1 < k_0 \leq a_2.
    \end{cases}
\end{equation}
From the definition of $\G_p$ in \eqref{eq:p-gamma-defn}, we immediately conclude the following translation between the Pochhammer symbol and $\G_p$ for all $0 \leq a \leq p-1$:
\begin{equation}\label{eq:poch-p-gamma}
    (t)_a = (-1)^a\frac{\G_p(t+a)}{\G_p(t)}(t+[-t]_0)^{\nu(a, [-t]_0)}.
\end{equation}
We now define a function $G_1(x):\Z_p\rightarrow\Z_p$ by
\begin{equation*}\label{eq:G1-defn}
    G_1(x)\colonequals \frac{\G_p'(x)}{\G_p(x)}.
\end{equation*} 
Taking logarithmic derivatives of both sides of \eqref{eq:poch-p-gamma} yields 
\begin{equation}\label{eq:poch-diff-G1}
    \frac{\frac{d}{dt}(t)_a}{(t)_a} = G_1(t+a) - G_1(t) + \frac{\nu(a,[-t]_0)}{t+[-t]_0}.
\end{equation}
We can also logarithmically differentiate the reflection identity \eqref{eq:p-gamma-reflect} to obtain
\begin{equation}\label{eq:G1-reflect}
    G_1(a) = G_1(1-a) \quad a \in \Z_p.
\end{equation}
It will now be useful for us to generalize the function $G_1$ in analogy with our notation in \eqref{eq:shorthand-quotient}.  Specifically, we set
\begin{equation*}\label{eq:G1-general-defn}
    G_1(\boldsymbol{\ba}+a, \boldsymbol{\bbeta}+a) \colonequals \mathrm{logder}\left(\G_p\left(\frac{\ba+a}{\bbeta+a}\right)\right) = \sum_{j=1}^n G_1(r_i+a)-G_1(q_i+a).
\end{equation*}  
We note that $ G_1(\boldsymbol{\ba}+a, \boldsymbol{\bbeta}+a)$ agrees with the definition of the function $J_1(\ba, \bbeta; a)$ in \cite{LTYZ} and in \cite{Allen}. 
The following  
result allows us to expand our hypergeometric functions $p$-adically:
\begin{theorem}[Long--Ramakrishna, \cite{LR}]\label{thm:p-gamma-approx}
    For $p \geq 5$, $r \in \N$ and $a, m \in \Z_p$,  we have
    \[
        \frac{\Gamma_p(a+mp^r)}{\Gamma_p(a)} \equiv 1+G_1(a)mp^r \pmod{p^{2r}}.
    \]
\end{theorem}
As an immediate consequence, we can evaluate the product modulo $p^2$ of the $\Gamma_p$ quotient in \eqref{eq:poch-quotient-to-p-gamma} as we run over all elements of $\ba$ and $\bbeta$.
\begin{corollary}
    For multi-sets $\ba = \left\{r_1, r_2, \hdots, r_n\right\}$, $\bbeta = \left\{q_1, q_2, \hdots q_n\right\}$, and fixed $0 \leq k_0 \leq p-1$, and $k_1 \in \Z_p$, we have
    \begin{equation}\label{eq:p-gamma-expand}
        \Gamma_p\left( \frac{\bbeta+k_0, \ba+k_0+k_1p}{\ba+k_0, \bbeta+k_0+k_1p}\right) \equiv 1+ G_1(\ba+k_0, \bbeta+k_0)k_1 p \pmod{p^2}.
    \end{equation}
\end{corollary}
Additionally, logarithmically differentiating \Cref{thm:p-gamma-approx} gives the congruence
\begin{equation}\label{eq:G1-p-shift}
    G_1(a+mp^r) \equiv G_1(a) \pmod{p^{r}}.
\end{equation} 
\subsection{Dwork type congruences}\label{ss:Dwork}
\subsubsection{The general case}
We use the shorthand notation $F_{s}(\ba,\bbeta;\l)$ for the truncation of the classical hypergeometric series after $p^{s}$ terms evaluated at $z=\l$.  The discussion below originated in Dwork's paper \cite{Dwork}, in which $\bbeta=\{1,\cdots,1\}$.  Dwork showed that when $F_1(\ba,\bbeta,\l)\neq 0\pmod p$, there exists a $p$-adic unit root $\mu_{\ba,\bbeta,\l,p}$ such that for each $s\ge 0$  
\begin{equation}\label{eq:Dwork-unit-root-congruence}
    F_{s+1}(\ba,\bbeta;\l)/F_s(\ba',\bbeta';\l^p)\equiv \mu_{\ba,\bbeta,\l,p}\pmod{p^{s+1}},
\end{equation}
where $\ba'=\{r_1',\cdots,r_n'\}$ is the image of $\ba$ under the Dwork dash operation \eqref{eq:dwork-dash-defn} and $\bbeta'$ is defined similarly.  See more recent discussions of Dwork crystals by Beukers and Vlasenko in \cite{BVI, BVII}.  We also note the similarity between \eqref{eq:Dwork-unit-root-congruence} and \eqref{eq:u_f}.  To extend Dwork's work to more general $\bbeta$, we add the assumption of $p\equiv 1\pmod{M(\HD)}$ so that each $r_i$ and $q_i$ is invariant under the Dwork dash operation by \Cref{lem:dwork-fixed}.  The proofs in this section rely heavily on the assumption $q_i > r_{i+2}$ for all $1 \leq i \leq n-2$, as this dictates the $p$-adic valuations of the hypergeometric coefficients as documented in \Cref{rem:lowest-weight} and \Cref{fig:p-slope}. 
\begin{example}\label{eg:LC-mu}
    Let $\HD=\HD_2$ as before.  For each odd prime $p$ and $\l\in \Z_p\setminus\{0,1\}$ such that $a_p(\l)=H_p(\HD_2;\l)\not\equiv 0\pmod p$, $\mu_{\HD_2,\l,p}$ is the root of $T^2-a_p(\l)T+p$ in $\Z_p$ that is congruent to $a_p(\l) \pmod p$. 
\end{example}\par
We now turn our attention towards the Dwork type congruences appearing in \Cref{rem:super-pieces}.
\begin{lemma}\label{lem:Dwork-congruence}
    Let $\ba = \left\{r_1, r_2, \hdots, r_n\right\}$ and $\bbeta = \left\{q_1, q_2, \hdots, q_{n-2}, 1, 1\right\}$ form a primitive hypergeometric datum $\HD$ with $n \geq 2$ such that $q_i > r_{i+2}$ for each $1 \leq i \leq n-2$.  Let $M=M(\HD)$ and  $p \equiv 1 \pmod{M}$.  Then for each $\l\in\Z_p$
    \[
        F_{s+1}(\ba, \bbeta; \l) - p \l^p F_s'(\ba, \bbeta; \l^{p}) E_{\mathrm{Dwork}}(\ba, \bbeta; \l) \equiv F_s (\ba, \bbeta; \l^{p}) F_1(\ba, \bbeta; \l) \pmod{p^2},
    \] 
    where here $F'$ denotes the derivative, not the Dwork dash operation, and
    \begin{singnumalign}\label{eq:EDwork}
        E_{\mathrm{Dwork}}(\ba, \bbeta; \l) \colonequals& \sum_{k=0}^{a_1} \hypcoeff{k} \l^{k} \left(G_1(\ba+k, \bbeta+k)\right) + \sum_{k=a_1+1}^{a_2} \hypcoeff{k} \frac{\l^{k}}{pr_1},
    \end{singnumalign} 
    where $a_i$ is defined as in \eqref{eq:ai-bi-defn}.
\end{lemma}
Under our assumptions, $E_{\mathrm{Dwork}}(\ba, \bbeta; \l) \in\Z_p$.  If  $F_1(\ba, \bbeta; \l) \not\equiv 0 \pmod{p}$,  meaning the ordinary case, then  $\l^p F'_s(\ba, \bbeta; \l^p) \in \Z_p$.  To show the desired mod $p^2$ supercongruences, we need to show that the $p$-linear term $F_1'(\ba,\bbeta;\l)E_{\mathrm{Dwork}}(\ba, \bbeta; \lambda)$ vanishes modulo $p$.  Our definition of $E_{\mathrm{Dwork}}$ is made because this term will always vanish for $\HD$ as in \Cref{thm:supercongruences} with $\gamma(\HD) \leq 1$ and $\lambda=1$.  In this case, we have the following:
\begin{remark}\label{rem:role-of-EDwork}
   Assume $\l\in\Z_p$ and $p$ is ordinary. 
   Then $F_1(\ba,\bbeta;\l)\equiv \mu_{HD,\l,p}\pmod{p^2}$ where $\mu_{HD,\l,p}$ is the unit root as in \eqref{eq:Dwork-unit-root-congruence} if $E_{\mathrm{Dwork}}(\ba, \bbeta; \l) \equiv 0\pmod{p}$.  
\end{remark}

We are now ready to prove \Cref{lem:Dwork-congruence}. 

\begin{proof}[Proof of \Cref{lem:Dwork-congruence}]
    We decompose our index of summation by $k = k_0+k_1p$ with $0 \leq k_0 \leq p-1$. 
    We find
    \begin{align*}
        F_{s+1}(\ba, \bbeta; \l) &= \sum_{k_1=0}^{p^s-1} \sum_{k_0=0}^{p-1} \frac{(\ba)_{k_0+k_1p}}{(\bbeta)_{k_0+k_1p}} \l^{k_0+k_1p} \\
        &\hspace{-2mm} \stackrel{\eqref{eq:poch-quotient-to-p-gamma}}{=} \sum_{k_1=0}^{p^s-1} \frac{(\ba)_{k_1}}{(\bbeta)_{k_1}} \l^{k_1 p} \sum_{k_0=0}^{p-1} \frac{(\ba)_{k_0}}{(\bbeta)_{k_0}} \l^{k_0} \Lambda_{\ba, \bbeta}(k_0+k_1p) \Gamma_p\left( \frac{\boldsymbol{\alpha}+k_0+k_1p, \boldsymbol{\bbeta}+k_1}{\boldsymbol{\alpha}+k_0, \boldsymbol{\bbeta}+k_0+k_1p}\right) \\
        & \hspace{-2mm} \stackrel{\eqref{eq:p-gamma-expand}}{\equiv} \sum_{k_1=0}^{p^s-1} \frac{(\ba)_{k_1}}{(\bbeta)_{k_1}} \l^{k_1 p} \sum_{k_0=0}^{p-1} \frac{(\ba)_{k_0}}{(\bbeta)_{k_0}} \l^{k_0} \\ 
        &\hspace{25mm} \times \left(1+G_1(\boldsymbol{\alpha}+k_0, \boldsymbol{\bbeta}+k_0)k_1p\right) \Lambda_{\ba, \bbeta}(k_0+k_1p) \pmod{p^2}. 
    \end{align*}
    As noted in \Cref{rem:lowest-weight} the hypergeometric coefficients $\hypcoeff{k}$ are always zero modulo $p^2$ for $k \geq a_2+1$.  Using this as well as \eqref{eq:Lambda-simple} then gives
    \begin{align*}
        F_{s+1}(\ba, \bbeta; \l) & \equiv \sum_{k_1=0}^{p^{s}-1} \frac{(\ba)_{k_1}}{(\bbeta)_{k_1}} \l^{k_1 p} \biggr[ \sum_{k_0=0}^{a_1} \frac{(\ba)_{k_0}}{(\bbeta)_{k_0}} \l^{k_0} \left(1 + G_1(\boldsymbol{\alpha}+k_0, \boldsymbol{\bbeta}+k_0)k_1p\right) \\
        &\hspace{30mm} + \sum_{k_0=a_1+1}^{a_2} \frac{(\ba)_{k_0}}{(\bbeta)_{k_0}} \l^{k_0} \left(1+\frac{k_1}{r_1}\right) \biggr] \pmod{p^2} \\&\equiv F_s(\ba, \bbeta; \l^p) F_1(\ba, \bbeta; \l) + p\l^p F_s'(\ba,\bbeta;\l^p) E_{\mathrm{Dwork}}(\ba, \bbeta; \l) \pmod{p^2}.
      \end{align*} 
\end{proof}
\subsubsection{The \texorpdfstring{$\l=1$}{lambda = 1} case}
We will adopt a strategy used in \cite{Allen,LTYZ,OSZ} by constructing a rational function $R(t)\in \Q(t)$---depending on $p$---such that the sum of the residues of $R(t)$ is congruent modulo $p$ to $E_{\mathrm{Dwork}}(\ba, \bbeta; 1)$.  Thus, we will be able to translate the corresponding residue sum identity into our desired congruence $E_{\mathrm{Dwork}}(\ba, \bbeta; 1) \equiv 0 \pmod{p}$.  Examining $E_{\mathrm{Dwork}}(\ba, \bbeta; 1)$ in \eqref{eq:EDwork} more closely, the desired rational function will satisfy the following  conditions when we consider its reduction $\bar R(t) \in \F_p(t)$:
\begin{enumerate}[label=(\arabic*)]
    \item For each $k\in\F_p$, $R(t)$ has a pole at $t=-k$.  If $n_k$ is the order of this pole, then $(R(t)(t+k)^{n_k})|_{t=-k}\equiv C \frac{(\ba)_k}{(\bbeta)_k}\pmod p$ for some constant $C$ independent of $k$.
    \item The order $n_k$ is equal to $2$ for $k \in [0,\hdots, a_1]$ and $1$ for $k \in [a_1+1,\hdots,a_2]$.  This will lead to the appearance of $G_1$ in the first sum defining $E_{\mathrm{Dwork}}$. 
\end{enumerate}
Additionally, we wish to balance the degrees of the numerator and denominator of $R(t)$ so that the residue at infinity is zero for $\gamma(\HD) < 1$ and becomes nonzero exactly when $\gamma(\HD) = 1$, as this will lead to the appearance of the $p$-linear term on the left-hand side of \eqref{eq:super-combined}.  The choice of such an $R(t)$ is not unique, but each choice will give the same first $p$-adic digit for the $E_{\mathrm{Dwork}}$ error term. 

We now define a  rational function that can be used to handle all cases listed in Lemma \ref{lem:Dwork-congruence}. We will need to track the number of elements of $\bbeta$ which are not equal to $1$, and so first we make the following definition:
\begin{defn}\label{defn:n-hat}
    For a fixed $\HD = \left\{\ba, \bbeta\right\}$ of length $n$, let $\hat{n}$ be the largest index $2 \leq \hat{n} \leq n$ such that $q_{\hat{n}} < 1$, or $0$ if no such index exists.
\end{defn}
We define a rational function by
\begin{equation}\label{eq:Rt-defn}
    R(t) = \frac{\prod\limits_{i=2}^n (t+1-r_i-p)_{p-a_i-1}}{(t)_{a_1+1}^2(t+a_1+1)_{b_1-a_1}(t+b_1+1)_{p-b_1-1}^2\prod\limits_{i=2}^{\hat{n}} (-t+1+ip)_{p-b_i-1}}.
\end{equation} 
We note that $R(t)$ is fully reduced, as the denominator has roots only at integers while the numerator only has roots at integer shifts of the $r_i$, which are non-integral. 
Our broad strategy will be to consider the sum of the residues of $R(t)$ two ways---over $\mathbb{C}$ we will use standard residue calculus to show that the sum of the residues is zero if and only if $\gamma(\HD) \le 1$, while $p$-adically we will show that the difference between this residue sum and $E_{\mathrm{Dwork}}(\ba, \bbeta; 1)$ will be small.  Together, this will yield our desired congruence $E_{\mathrm{Dwork}}(\ba, \bbeta; 1) \equiv 0 \pmod{p}$.  Our particular choice of $R(t)$ in \eqref{eq:Rt-defn} is made to make both sides of this calculation as simple as possible.  For example, the shifts of the Pochhammer symbols by $ip$ in the denominator are not strictly necessary; we make this choice as it guarantees that the poles at each $t = k+ip$ with $2 \leq i \leq n$ and $1 \leq k \leq p-b_i-1$ are simple.  The perturbation in the numerator by $-r_i-p$ is made so that the residues of these simple poles will be divisible by $p$.  The function $R(t)$ has partial fraction decomposition
\begin{singnumalign}\label{eq:R-partial}
    R(t) =& \sum_{k=0}^{a_1} \left(\frac{A_k}{(t+k)^2} + \frac{B_k}{t+k} \right) + \sum_{k=a_1+1}^{b_1} \frac{B_k}{t+k} + \sum_{k=b_1+1}^{p-1} \left(\frac{A_k}{(t+k)^2} + \frac{B_k}{t+k} \right) \\
    &+ \sum_{i=2}^{\hat{n}} \sum_{k=1}^{p-b_i-1} \frac{C_{i,k}}{-t+k+ip}.
\end{singnumalign}
Going forward, we will separate the range $[0,\hdots,p-1]$ into four intervals, as illustrated in \Cref{fig:p-slope}.  Namely, we set
\begin{singnumalign}\label{eq:interval-defn}
    I_1 &= [0,\hdots,a_1], &I_2 &= [a_1+1,\hdots,a_2] \\
    I_3 &= [a_2+1, \hdots, b_1], &I_4 &= [b_1+1,\hdots,p-1].
\end{singnumalign}
We will eventually relate this decomposition to $E_{\mathrm{Dwork}}(\ba, \bbeta; 1)$ by showing that the sum of the residues $B_k$ over $I_1$ and $I_2$ each match to the corresponding sums in the definition of $E_{\mathrm{Dwork}}$.  Before doing so, we show that the remaining residues vanish modulo $p$ and so will not contribute $p$-adically.
\begin{lemma}\label{lem:beta-hat-residues}
    Let $\HD, M$, and $p$ be as in \Cref{lem:Dwork-congruence}, with the additional assumption that $p > \hat{n}+1$ where $\hat{n}$ is defined as in \Cref{defn:n-hat}.  With $C_{i,k}$ as in \eqref{eq:R-partial} and for all $2 \leq i \leq \hat{n}$ and $1 \leq k \leq p-b_i-1$ we have
    \[
        C_{i,k} \equiv 0 \pmod{p}.
    \]
\end{lemma}
\begin{proof}
    For each fixed pair $(i_0, k_0)$, the corresponding residue $C_{i_0,k_0}$ can be computed directly as
    \begin{align*}
        C_{i_0,k_0} &= (-t+k_0+i_0p)R(t) \biggr\vert_{t = k_0+i_0p} \\
        =& \frac{\prod\limits_{i=2}^n(k_0+1-r_i+(i_0-1)p)}{(k_0+i_0p)_{a_1+1}^2(k_0+a_1+1+i_0p)_{b_1-a_1}(k_0+b_1+1+i_0p)_{p-b_1-1}^2} \\
        &\times \frac{1}{(-k_0+1)_{k_0-1}(1)_{p-b_{i_0}-k_0-1} \prod\limits_{\substack{i=1 \\ i \ne i_0}}^{\hat{n}} (-k_0+1+(i-i_0)p)_{p-b_i-1}}.
    \end{align*}
    The index of each Pochhammer symbol above is less than or equal to $p$, and so each symbol contains at most one multiple of $p$.  As we saw in \eqref{eq:poch-p-gamma}, for $0 \leq a \leq p-1$, $p \mid (t)_a$ if and only if the index $a$ is larger than the first $p$-adic digit of the negative of the argument, $[-t]_0$.  In this case, the multiple of $p$ appearing in $(t)_a$ is exactly $t+[-t]_0$.  We can use this term by term to evaluate the $p$-adic valuation of the above expression for $C_{i_0,k_0}$.  First we note that, together,
    \[
        (k_0+i_0p)_{a_1+1}(k_0+a_1+1+i_0p)_{b_1-a_1}(k_0+b_1+1+i_0p)_{p-b_1-1}
    \]
    is a product over all $(k_0+i_0p+j)$ with $j$ ranging from $0$ to $p-1$.  In particular, exactly one of these three Pochhammer symbols is divisible by $p$, and the $p$-divisible term is exactly
    \[
        (k_0+i_0p)+[-k_0-i_0p]_0 = (k_0+i_0p)+(p-k_0) = (i_0+1)p.
    \]
    Because $0 < i_0 +1 \leq \hat{n}+1 < p$, the $p$-adic valuation of $(i_0+1)p$ is exactly one.  We could be precise about which of the three Pochhammer symbols this term will appear in for our fixed $k_0$, but it will be enough for our purposes to conclude that
    \begin{equation}\label{eq:a1-b1-valuation}
        \mathrm{ord}_p\left((k_0+i_0p)_{a_1+1}^2(k_0+a_1+1+i_0p)_{b_1-a_1}(k_0+b_1+1+i_0p)_{p-b_1-1}^2\right) \leq 2,
    \end{equation}
    where we use $\mathrm{ord}_p(\cdot)$ to denote the $p$-adic valuation.  Both of the terms $(-k_0+1)_{k_0-1} = (-1)^{k_0-1}(k_0-1)!$ and $(1)_{p-b_{i_0}-k_0-1} = (p-b_{i_0}-k_0-1)!$ have $p$-adic valuation zero.  Thus, we now need only consider
    \[
        \frac{\prod\limits_{i=2}^n(k_0+1-r_i+(i_0-1)p)}{\prod\limits_{\substack{i=1 \\ i \ne i_0}}^{\hat{n}} (-k_0+1+(i-i_0)p)_{p-b_i-1}}.
    \]
    In the denominator, the first $p$-adic digit of the negative of the argument is
    \[
        [k_0-1]_0 = k_0-1.
    \]
    This is smaller than the index if and only if $k_0 < p-b_i$, in which case the multiple of $p$ appearing is $(i-i_0)p$ which has valuation $1$ as $-p < i-i_0 < p$ and $i \ne i_0$.  Let $1 \leq \ell \leq \hat{n}$ be the largest index such that $k_0 < p-b_\ell$.  Then   %
    \begin{equation}\label{eq:denom-valuation}
        \mathrm{ord}_p\left(\prod\limits_{\substack{i=2 \\ i \ne i_0}}^{\hat{n}} (-k_0+1+(i-i_0)p)_{p-b_i-1}\right) = \ell-2.
    \end{equation}
    The $-2$ arises from the omitted terms at $i = 1$ and $i = i_0$. \par
    The numerator behaves quite similarly.  By an analogous argument, we find that 
    \[
        p \mid (k_0+1-r_i+(i_0-1)p)_{p-a_i-1}
    \]
    if and only if $k_0 < p-a_i-1$, in which case the multiple of $p$ appearing is $p(i_0-r_i)$.  We can not guarantee that $i_0-r_i$ has $p$-adic valuation zero, but it must have a nonnegative valuation which will be sufficient.  As $k_0 < p-b_i$, the inequality $k_0 < p-a_i-1$ must hold for all $i$ such that $a_i < b_\ell$.  By hypothesis, $a_i < b_{\ell}$ holds for all $1 \leq i \leq 2+\ell$.  Therefore,
    \begin{equation}\label{eq:num-valuation}
        \nu_p\left(\prod_{i=2}^n (k_0+1-r_i+(i_0+1)p)_{p-a_i-1}\right) \geq \ell+1.
    \end{equation}
    Combining \eqref{eq:a1-b1-valuation}, \eqref{eq:denom-valuation}, and \eqref{eq:num-valuation} then yields
    \[
        \mathrm{ord}_p(C_{i_0,k_0}) \geq \ell+1 - 2 - (\ell-2)= 1,
    \]
    as was to be shown.
\end{proof}
We will see in the proof of \Cref{prop:Dwork-error} that for $k \in I_3$ we have $B_k \equiv 0 \pmod{p}$.  For now we move past $I_3$ and consider the residues on the final interval $I_4$.
\begin{lemma}\label{lem:I3-residues}
    Let $\HD, M$, and $p$ be as in \Cref{lem:beta-hat-residues}.  In the notation of \eqref{eq:R-partial} and \eqref{eq:interval-defn}, $B_k \equiv 0 \pmod{p}$ for all $k \in I_4$.
\end{lemma}
\begin{proof}
The proof is essentially a simpler version of the argument we just made to show \Cref{lem:beta-hat-residues}, with the one complication that our poles now have order $2$.  Thus,
\begin{equation}\label{eq:Bk-I3}
    B_k = \frac{d}{dt}(t+k)^2R(t) \biggr\vert_{t=-k} = (t+k)^2R(t) \frac{\frac{d}{dt}(t+k)^2R(t)}{(t+k)^2R(t)} \biggr\vert_{t=-k}.
\end{equation}
We rewrite $B_k$ in this manner as it will be simpler to evaluate the $p$-adic valuations of the terms arising from $(t+k)^2R(t)$ and its logarithmic derivative separately.  In particular, we note that logarithmic differentiation transforms the Pochhammer symbol $(t)_a$ into
\[
    \frac{\frac{d}{dt}(t)_a}{(t)_a} = \sum_{i=0}^{a-1} \frac{1}{t+i}.
\]
Therefore, the $p$-adic valuation of the logarithmic derivative, after specializing $t=-k$, will be bound below by the negative of the largest exponent appearing on any multiple of $p$ in the original Pochhammer symbol.  \par
We turn our attention to the $p$-adic valuation of
\begin{align*}
    (t+k)^2R(t)\biggr\vert_{t=-k} =& \frac{\prod_{i=2}^n (-k+1-r_i-p)_{p-a_i-1}}{(-k)_{a_1}^2(-k+a_1+1)_{b_1-a_1}(-k+b_1+1)_{k-b_1-1}^2(1)_{p-k-1}^2} \\
    &\times  \frac{1}{\prod_{i=2}^{\hat{n}}(k+1+ip)_{p-b_i-1}}.
\end{align*}
This valuation can be computed via the same direct approach we used in the proof of \Cref{lem:beta-hat-residues}, and the results are quite similar.  As such, we omit many of the details.  First, we find
\[
    \mathrm{ord}_p((-k)_{a_1}^2(-k+a_1+1)_{b_1-a_1}(-k+b_1+1)_{k-b_1-1}^2(1)_{p-k+1}^2) = 0,
\]
as the one guaranteed multiple of $p$ is at the residue $t+k$ that we are removing.  Next, the terms $(k+1+ip)_{p-b_i-1}$ contain a multiple of $p$, namely $p(i+1)$ which has valuation $1$, if and only if $k > b_i$.  Thus, letting $2 \leq \ell \leq \hat{n}$ be the maximal index for which $k > b_\ell$, we find
\[
    \mathrm{ord}_p\left(\prod_{i=2}^{\hat{n}}(k+1+ip)_{p-b_i-1}\right) = \ell-1.
\]
The terms $(-k+1-r_i-p)_{p-a_1-1}$ appearing in the numerator will be divisible by $p$ if and only if $k > a_i$.  As we saw in the proof of the previous Lemma, our hypotheses ensure this holds for all $1 \leq i \leq 2+\ell$.  The multiple of $p$ which does appear in this case is $-p(1+r_i)$, which has $p$-adic valuation at least $1$.  As this valuation could a priori be larger than $1$, we define
\[
    m_p = \max_{1 \leq i \leq n} \mathrm{ord}_p(1+r_i).
\]
We conclude
\[
    \mathrm{ord}_p\left(\prod_{i=2}^n (-k+a_i+1-p)_{p-a_i-1}\right) \geq \ell+1+m_p.
\]
Therefore,
\[
    \mathrm{ord}_p\left((t+k)^2R(t) \biggr\vert_{t=-k}\right) \geq 2+m_p.
\]
Additionally, we have seen that only linear powers of $p$ appear in the denominator $(t+k)^2R(t)$, and so the maximal exponent of $p$ appearing overall will be $1+m_p$.  As noted above, this guarantees that
\[
    v_p\left(\frac{\frac{d}{dt}(t+k)^2R(t)}{(t+k)^2R(t)}\right) \geq -1-m_p.
\]
Taken together, $v_p(B_k) \geq 1$, as was to be shown.
\end{proof}
\begin{proposition}\label{prop:Dwork-error}
    With $\HD, \, M$, and $p$ as in \Cref{lem:beta-hat-residues} and $E_{\mathrm{Dwork}}$ defined as in \Cref{lem:Dwork-congruence}, we have
    \[
        E_{\mathrm{Dwork}}(\ba, \bbeta; 1) \equiv  (-1)^{\sum_{i=2}^{\hat{n}} b_i} \G_p\left(\frac{\bbeta}{\ba}\right) \mathrm{Res}_{t=\infty}(R(t)) \pmod{p},
    \] where $\G_p\left(\frac{\bbeta}{\ba}\right) \colonequals \prod_{i=1}^n \frac{\G_p(q_i)}{\G_p(r_i)}.$
\end{proposition}
\begin{proof}
    When $\l = 1$, the error term \eqref{eq:EDwork} reduces to
    \[
        E_{\mathrm{Dwork}}(\ba, \bbeta; 1) = \sum_{k_0=0}^{a_1} \hypcoeff{k_0} G_1(\boldsymbol{\alpha}+k_0, \boldsymbol{\beta}+k_0) + \sum_{k_0=a_1+1}^{a_2} \hypcoeff{k_0} \frac{1}{p r_1}.
    \]
    With our rational function $R(t)$ expressed in the partial fraction decomposition \eqref{eq:R-partial}, the Residue Theorem then implies
    \begin{equation}\label{eq:res-sum}
        \sum_{k=0}^{p-1} B_k + \sum_{i=2}^{\hat{n}}\sum_{k=1}^{p-b_i-1} C_{i,k} = -\mathrm{Res}_{t = \infty} (R(t)).
    \end{equation}
    We have already considered the $p$-adic behavior of $C_{i,k}$ as well as $B_k$ for $k \in I_4$, so we now turn our attention to $B_k$ for $k \in I_1$, $I_2$, and $I_3$.  
    
    First, for $k \in I_1$, $A_k$ can be computed directly as
    \begin{equation}\label{eq:Ak-direct}
        A_k = \frac{\prod_{i=2}^n (-k+1-r_i-p)_{p-a_i-1}}{(k!)^2(a_1-k)!^2(-k+a_1+1)_{b_1-a_1}(-k+b_1+1)_{p-b_1-1}^2 \prod_{i=2}^{n}(k+1+ip)_{p-b_i-1}}.
    \end{equation}
    Note that we have changed the index on the product in the denominator from $\hat{n}$ to $n$.  Either indexing gives the same value, as $q_i = 1$ implies $b_i=p-1$ and the corresponding Pochhammer symbol is equal to $1$.  Reintroducing those $q_i$ which are equal to $1$ now simplifies our organization throughout the remainder of the proof.  We interpret $A_k$ $p$-adically using \eqref{eq:poch-p-gamma}. 
    For example,
    \begin{align*}
        (-k+1-r_i-p)_{p-a_i-1} &= (-1)^{a_i} \frac{\G_p(-k-pr_i)}{\G_p(-k+1-r_i-p)} \left(-pr_i\right)^{\nu(p-a_i-1, p-a_i-1+k)} \\
        &= (-1)^{a_i} \frac{\G_p(-k-pr_i)}{\G_p(-k+1-r_i-p)}.
    \end{align*}
    In fact, the $\nu$ terms are zero for each of our Pochhammer symbols, and we are left with only the signs and the $\G_p$ values.  Following a similar analysis for each of the remaining Pochhammer symbols in $A_k$, we find
    \[
        A_k = \frac{(-1)^{\sum_{i=1}^n a_i+b_i} \G_p(-k+b_1+1)\prod_{i=2}^{n}\G_p(-k-pr_i)\G_p(k+1+ip)}{\G_p^2(k+1)\G_p(a_1-k+1)\G_p^2(p-k)\prod_{i=2}^n \G_p(-k+1-r_i-p)\G_p(k-b_i+(1+i)p)}
    \]
    This expression can be simplified extensively using \eqref{eq:p-gamma-reflect}, \eqref{eq:poch-p-gamma}, and \eqref{eq:p-gamma-expand}.  Namely,
    \begin{align*}
        A_k &\hspace{-1mm}\stackrel{\eqref{eq:p-gamma-reflect}}{=} (-1)^{\sum_{i=2}^n b_i} \frac{ \G_p(k-a_1)\G_p^2(k+1-p)\prod_{i=2}^n\G_p(k+r_i+p)\G_p(k+1+ip)}{\G_p^2(k+1)\G_p(k-b_1) \prod_{i=2}^n \G_p(k+1+pr_i)\G_p(k-b_i+(1+i)p)} \\
        &\hspace{-2mm} \stackrel{\eqref{eq:poch-p-gamma}}{\equiv} (-1)^{\sum_{i=2}^n b_i} \hypcoeff{k} \G_p\left(\frac{\ba}{\bbeta}\right) \pmod{p}.
    \end{align*}
    Next, we compute $B_k$ using the residue formula

    \begin{align*}\label{eq:Bk-res-defn}
        B_k &= \lim_{t \to -k} \frac{d}{dt}\left((t+k)^2R(t)\right) 
        = A_k \lim_{t \to -k} \frac{\frac{d}{dt}\left((t+k)^2R(t)\right)}{(t+k)^2R(t)}.
    \end{align*}
    Therefore,
    \begin{align*}
        \frac{B_k}{A_k} &= \sum_{i=2}^n \sum_{j=1}^{p-a_i-1} \frac{\frac{d}{dt}(t+1-r_i-p)_{p-a_i-1}}{(t+1-r_i-p)_{p-a_i-1}} - 2 \frac{\frac{d}{dt}(t)_k}{(t)_k} - 2 \frac{\frac{d}{dt}(t+k+1)_{a_1-k}}{(t+k+1)_{a_1-k}} \\
        &\hspace{10mm} - \frac{\frac{d}{dt}(t+a_1+1)_{b_1-a_1}}{(t+a_1+1)_{b_1-a_1}} - 2 \frac{\frac{d}{dt}(t+b_1+1)_{p-b_1-1}}{(t+b_1+1)_{p-b_1-1}} - \sum_{i=2}^n \frac{\frac{d}{dt}(-t+1+ip)_{p-b_i-1}}{(-t+1+ip)_{p-b_i-1}} \biggr\vert_{t=-k}.
    \end{align*}
    Each of these logarithmic derivatives can be evaluated using \eqref{eq:poch-diff-G1}---as in the computation of $A_k$ each of the $\nu$ terms will be zero.  Then, the $G_1$ terms are simplified using \eqref{eq:G1-reflect} so that each term has a positive $k$ and then reduced modulo $p$, using \eqref{eq:G1-p-shift} to remove each multiple of $p$ appearing inside of a $G_1$.  For example,
    \begin{align*}
        \frac{\frac{d}{dt}(t+1-r_i-p)_{p-a_i-1}}{(t+1-r_i-p)_{p-a_i-1}} 
        & \equiv G_1(k+1)-G_1(k+r_i) \pmod{p}.
    \end{align*}
    After computing an analogous reduction 
    we find
    \begin{equation}\label{eq:Bk-Ak-ratio-I1}
        \frac{B_k}{A_k} \equiv -G_1(\ba+k,\bbeta+k) \pmod{p}.
    \end{equation}
    Thus,
    \begin{equation*}\label{eq:Bk-I1}
        B_k \equiv (-1)^{1+\sum_{i=2}^n b_i}\G_p\left(\frac{\ba}{\bbeta}\right) \hypcoeff{k} G_1(\ba+k,\bbeta+k) \pmod{p}.
    \end{equation*}
    We now turn our attention to the interval $I_2 \cup I_3$.  Suppose $a_1+1 \leq k \leq b_1$.  Then
    \begin{align*}
        B_k &= (t+k)R(t) \biggr\vert_{t=-k} \\
        &= \frac{\prod_{i=2}^n (-k+1-r_i-p)_{p-a_i-1}}{(-k)_{a_1+1}^2(-k+a_1+1)_{k-a_1-1}(b_1-k)!(-k+b_1+1)_{p-b_i-1}^2 \prod\limits_{i=2}^{n}(k+1+ip)_{p-b_i-1}}
    \end{align*}
    We observe that the denominator never contains a multiple of $p$ as $k \leq b_1$, whereas the numerator will contain at least one multiple of $p$ for all $k \geq a_2$.  Thus, $B_k \equiv 0 \pmod{p}$ for all $k  \in I_3$.  On $I_2$, all of the $\nu$ terms arising from \eqref{eq:poch-p-gamma} will be zero and we can reduce in the same manner as we used for $A_k$ above.  After doing so, we find that for $a_1+1 \leq k \leq a_2$,
    \begin{equation}\label{eq:Bk-I2}
        B_k \equiv \frac{(-1)^{1+\sum_{i=2}^n b_i}}{pr_1} \hypcoeff{k} \G_p\left(\frac{\ba}{\bbeta}\right) \pmod{p}.
    \end{equation}
    The $1/(pr_1)$ arises from the fact that $k>a_1$ implies
    \[
        \G_p(r_1+k) = \frac{(-1)^k}{pr_1} \G_p(r_1)(r_1)_k.
    \]
    We now return to the residue sum $\eqref{eq:res-sum}$, which we reduce modulo $p$.  By \Cref{lem:beta-hat-residues}, \Cref{lem:I3-residues}, and the fact that $B_k \equiv 0 \pmod{p}$ for $k \in I_3$, we have
    \begin{align*}
        -\mathrm{Res}_{t=\infty}(R(t)) &\equiv \sum_{k=0}^{a_2} B_k \pmod{p} \\
        &\equiv- (-1)^{\sum_{i=2}^n b_i} \G_p\left(\frac{\ba}{\bbeta}\right) \biggr(\sum_{k=0}^{a_1} \hypcoeff{k}G_1(\ba+k, \bbeta+k) \\
        &\hspace{20mm} + \sum_{k=a_1+1}^{a_2} \hypcoeff{k} \frac{1}{pr_1} \biggr) \pmod{p} \\
        &\equiv- (-1)^{\sum_{i=2}^n b_i} \G_p\left(\frac{\ba}{\bbeta}\right)E_{\mathrm{Dwork}}(\ba,\bbeta;1) \pmod{p}.
    \end{align*}
    This completes the proof.
\end{proof}
In light of \Cref{rem:role-of-EDwork}, we are particularly interested in cases where this error term $E_{\mathrm{Dwork}}(\ba, \bbeta; 1)$ vanishes modulo $p$, which will occur 
when $R(t)$ is holomorphic at $\infty$. Recalling $\gamma(\HD)=-1+\sum_{i=1}^n(q_i-r_i)$, this vanishing of $E_{\mathrm{Dwork}}$ can be rephrased purely arithmetically as follows:
\begin{corollary}\label{cor:Dwork-super}
    With $\ba, \bbeta, M$, and $p$ as in \Cref{lem:beta-hat-residues}, if $\gamma(\HD) \le 1$ we have 
    \[
        F_{s+1}(\ba, \bbeta; 1) \equiv F_s (\ba, \bbeta; 1) F_1(\ba, \bbeta; 1) \pmod{p^2}.
    \]
\end{corollary}
\begin{proof}
    We assume that $\gamma(\HD) \le 1$, namely $\sum_{i=1}^n q_i-r_i\le 2$,  and multiply both sides by $p-1$.  Recalling that $a_i = (p-1)r_i$ and $b_i = (p-1)q_i$, our inequality then becomes
    \begin{align*}
        2p + \sum_{i=1}^n (a_i-b_i) &\geq 2.
    \end{align*}
    It is easy to check that the left-hand side of the above inequality is exactly the degree of the denominator of $R(t)$ minus the degree of the numerator, and so this final inequality 
    implies $\mathrm{Res}_{t=\infty}(R(t))$ being zero and hence for the error term $E_{\mathrm{Dwork}}(\ba, \bbeta; 1)$ to vanish modulo $p$.
\end{proof}
\subsection{Gross--Koblitz type supercongruences}\label{ss:GK}
\subsubsection{The general case}
We now consider the Gross--Koblitz type supercongruence stated in \Cref{rem:super-pieces}.  Recall that we have a direct relationship between the $p$-adic gamma function and the Gauss sum from 
\Cref{thm:gk}.  This can be extended to the following relationship between $H_p$ and truncated hypergeometric functions.
\begin{lemma}\label{lem:GK-congruence}
    Let $\ba, \bbeta$, and $M$ be as in \Cref{lem:Dwork-congruence}.  For a fixed prime $p \equiv 1 \pmod{M}$ and with $\omega_p$ the Teichm\"{u}ller character of $\mathbb{F}_p$, we have
    \[
        H_p(\ba, \bbeta; \l; \bar{\omega}_p) - E_{\mathrm{GK}}(\ba, \bbeta; \l^p)p \equiv F (\ba, \bbeta; \l^p)_{p-1} \pmod{p^2},
    \] 
    where
    \begin{equation}\label{eq:E(GK)}
         E_{\mathrm{GK}}(\ba, \bbeta; \l) \colonequals \sum_{k=0}^{a_1} \hypcoeff{k} \l^k \left(G_1(\boldsymbol{\alpha}+k, \boldsymbol{\bbeta}+k)k+1\right) + \sum_{k=a_1+1}^{a_2} \hypcoeff{k} \l^k \frac{k}{pr_1}.
    \end{equation}
\end{lemma}
\begin{remark}\label{rem:geom-error-terms} 
    Formally,
    \begin{equation}
        E_{\mathrm{GK}}(\ba, \bbeta; \l)-\l\frac{d}{d\l} E_{\mathrm{Dwork}}(\ba, \bbeta; \l)= \sum_{k=0}^{a_1} \hypcoeff{k} \l^k.
    \end{equation}
\end{remark}

\begin{proof}[Proof of \Cref{lem:GK-congruence}]
    As before, we let $1 \leq \hat{n} \leq n$ be the minimal index such that $q_{\hat{n}} \ne 1$.  If no such $\hat{n}$ exists, instead take $\hat{n}=0$.  By definition,
    \begin{align*}
        H_p(\ba, \bbeta; \lambda; \bar{\omega}_p) &= \frac{1}{1-p} \sum_{k=0}^{p-2} \omega_p^{-k}\left((-1)^n\lambda\right) \left(\prod_{j=1}^n \frac{\mathfrak{g}\left(\omega_p^{-(p-1)r_j-k}\right)}{\mathfrak{g}\left(\omega_p^{-(p-1)r_j}\right)}\right) \\
        &\hspace{30mm} \times \prod_{i=1}^{\hat{n}} \frac{\mathfrak{g}\left(\omega_p^{(p-1)q_i+k}\right)}{\mathfrak{g} \left(\omega_p^{(p-1)q_i}\right)} \left(-\mathfrak{g}\left( \omega_p^{p-1-k}\right)\right)^{n-\hat{n}}.
    \end{align*}
    First, we reverse the order of summation for $k>0$ by replacing $k$ by $p-1-k$.  This gives us
    \begin{align*}
        H_p(\ba, \bbeta; \lambda; \bar{\omega}_p) &= \frac{1}{1-p} \sum_{k=0}^{p-2} \omega_p^k\left((-1)^n\lambda\right) \left(\prod_{j=1}^n \frac{\mathfrak{g}\left(\omega_p^{(p-1)(1-r_j)+k}\right)}{ \mathfrak{g}\left(\omega_p^{(p-1)(1-r_j)}\right)}\right) \\
        &\hspace{30mm} \times \prod_{i=1}^{\hat{n}} \frac{\mathfrak{g}\left(\omega_p^{-(p-1)(1-q_i)-k}\right)}{\mathfrak{g} \left(\omega_p^{-(p-1)(1-q_i)}\right)} \left(- \mathfrak{g} \left( \omega_p^{k} \right) \right)^{n-\hat{n}}.
    \end{align*}
    We use the Gross--Koblitz formula (\Cref{thm:gk}) to rewrite the above expression in terms of $\Gamma_p$.  To do so we must have our exponent written in the range $[2-p,0]$, and so a shift by $p-1$ may be required if the exponent does not fall into this range.  With this in mind and recalling that $a_j = (p-1)r_j$,
    \begin{singnumalign}\label{eq:alpha-gauss-num}
        \g\left( \omega_p^{(p-1)(1-r_j)+k} \right) &=
        \begin{cases}
            \g\left(\omega_p^{(p-1)(-r_j)+k}\right) & 0 \leq k \leq a_j; \\
            \g\left(\omega_p^{(p-1)(-1-r_j)+k}\right) & a_j < k
        \end{cases} \\
        &=
        \begin{cases}
            -\pi_p^{r_j(p-1)-k} \G_p\left(\frac{r_j(p-1)-k}{p-1}\right) & 0 \leq k \leq a_j; \\
            -\pi_p^{(1+r_j)(p-1)-k} \G_p\left( \frac{(1+r_j)(p-1)-k}{p-1} \right) & a_j < k.
        \end{cases}
    \end{singnumalign}
    We can compute $\mathfrak{g}\left(\omega_p^{(p-1)(1-r_j)}\right)$ similarly.  Taking the quotient and simplifying we find
    \[
        \frac{\g\left(\omega_p^{(p-1)(1-r_j)+k}\right)}{\g\left(\omega_p^{(p-1)(1-r_j)}\right)} = (-p)^{\nu(k, a_j)}\pi_p^{-k} \frac{\G_p\left(\frac{r_j(p-1)-k}{p-1} + \nu(k,a_j) \right)}{\G_p(r_j)},
    \]
    where $\nu$ is defined as in \eqref{eq:nu-defn}.   An analogous computation shows that
    \begin{equation*}\label{eq:q1-gauss-full}
        \frac{\g\left(\omega_p^{-(p-1)(1-q_i)-k}\right)}{\g \left(\omega_p^{-(p-1)(1-q_i)}\right)} = (-1)^{b_i+1}(-1/p)^{\nu(k+1,b_i)} \pi_p^k \Gamma_p\left( \frac{k}{p-1} - q_i + \nu(b_i,k)\right)\Gamma_p(q_i)
    \end{equation*}
    and
    \begin{equation*}\label{eq:1-gauss-full}
        -\mathfrak{g}(\omega_p^{-k}) = \pi_p^k \G_p\left(\frac{k}{p-1}\right).
    \end{equation*}
    All together we have
    \begin{singnumalign}\label{eq:Hp-after-GK}
        H_p &= \frac{1}{1-p}\sum_{k=0}^{p-2} \omega_p^k((-1)^n\lambda) \left( \prod_{j=1}^n (-p)^{\nu(k,a_j)} \frac{\Gamma_p\left( r_j - \frac{k}{p-1} + \nu(k,a_j)\right)}{\Gamma_p(r_j)} \right) \\
        &\times \prod_{i=1}^{\hat{n}}\left((-1)^{b_i+1}(-1/p)^{\nu(k+1,b_i)} \Gamma_p\left( \frac{k}{p-1} - q_i + \nu(b_i,k)\right)\Gamma_p(q_i)\right) \left( \Gamma_p\left(\frac{k}{p-1}\right)\right)^{n-\hat{n}}.
    \end{singnumalign}
    Using the functional equation \eqref{eq:p-gamma-shift} as well as the definition of $\nu$ we have
    \begin{equation*}\label{eq:Hp-alpha-gamma-term}
        (-p)^{\nu(k,a_j)}\Gamma_p\left( r_j - \frac{k}{p-1} + \nu(k,a_j)\right) = \left(p\left(r_j - \frac{k}{p-1}\right)\right)^{\nu(k,a_j)}\Gamma_p\left(r_j - \frac{k}{p-1}\right).
    \end{equation*}
    In the product we have
    \begin{singnumalign}\label{eq:Hp-alpha-gamma-product}
        \prod_{j=1}^n (-p)^{\nu(k,a_j)}\Gamma_p\biggr(r_j - &\frac{k}{p-1} + \nu(k,a_j)\biggr) \\
        &= \prod_{j=1}^n \left(p \left(r_j - \frac{k}{p-1}\right)\right)^{\nu(k,a_j)}\Gamma_p\left(r_j - \frac{k}{p-1}\right),
    \end{singnumalign}
    and in particular this vanishes modulo $p^2$ for $k \geq a_2$. \\
    Turning our attention to the $q_i$ term, \eqref{eq:p-gamma-reflect} yields
    \begin{equation*}
        \Gamma_p\left( \frac{k}{p-1} - q_i + \nu(b_i,k)\right) = \frac{(-1)^{x_i}}{\Gamma_p\left(q_i - \frac{k}{p-1} +\nu(k+1,b_i)\right)},
    \end{equation*}
    where
    \begin{equation}\label{eq:x0-exp-defn}
        x_i = 
        \begin{cases}
            \nu(b_i,k)-k+b_i & \text{ if } 0 < \nu(b_i, k) - k + b_i; \\
            \nu(b_i, k) + p - k + b_i & \text{otherwise}.
        \end{cases}
    \end{equation}
    From this, \eqref{eq:p-gamma-shift}, and the fact that $1-\nu(b_i,k) = \nu(k+1,b_i)$ we conclude
    \begin{singnumalign}\label{eq:Hp-beta-gamma-term}
        \left( -p \right)^{-\nu(k+1,b_i)} \Gamma_p \biggr( &\frac{k}{p-1} - q_i + \nu(b_i,k) \biggr) \\[2mm]
        &= \frac{(-1)^{\nu(k+1,b_i)+x_i}}{p^{\nu(k+1,b_i)} \Gamma_p\left(q_i - \frac{k}{p-1} + \nu(k+1, b_i) \right)} \\
        &= \frac{(-1)^{x_i}}{\left(p\left(q_i - \frac{k}{p-1}\right)\right)^{\nu(k+1,b_i)} \Gamma_p\left(q_i - \frac{k}{p-1}\right)}.
    \end{singnumalign}
    Additionally,
    \begin{equation}\label{eq:Hp-1-gamma-term}
        \G_p\left(\frac{k}{p-1}\right) = \frac{(-1)^k}{\G_p\left(1-\frac{k}{p-1}\right)} = (-1)^{k+1}\frac{\G_p(1)}{\G_p\left(1-\frac{k}{p-1}\right)}.
    \end{equation}
    Combining \eqref{eq:Hp-beta-gamma-term} and \eqref{eq:Hp-1-gamma-term} yields
    \begin{singnumalign}\label{eq:Hp-beta-gamma-final}
        &\prod_{i=1}^{\hat{n}} (-1)^{b_i+1}\left( -p \right)^{-\nu(k+1,b_i)}\G_p\left(\frac{k}{p-1}-q_i+\nu(b_i,k)\right)\G_p(q_i) \G_p^{n-\hat{n}}\left(\frac{k}{p-1}\right) \\
        &= \prod_{i=1}^{\hat{n}} \frac{(-1)^{x_i+b_i+1}\G_p(q_i)}{\left(p\left(q_i-\frac{k}{p-1}\right)\right)^{\nu(k+1,b_i)}\G_p\left(q_i-\frac{k}{p-1}\right)} \left(\frac{(-1)^{k+1}\G_p(1)}{\G_p\left(1-\frac{k}{p-1}\right)}\right)^{n-\hat{n}}. \\
    \end{singnumalign}
    Using \eqref{eq:Hp-alpha-gamma-product} and \eqref{eq:Hp-beta-gamma-final}, we can further reduce \eqref{eq:Hp-after-GK} to
    \begin{singnumalign}\label{eq:Hp-full-gamma-final}
        H_p(\ba, \bbeta; \l; \bar{\omega}_p) =& \frac{1}{1-p} \sum_{k=0}^{p-2} (-1)^{\hat{n}+\sum_{i=1}^{\hat{n}}(b_i+x_i)+(n-\hat{n})(k+1)} \omega_p^k((-1)^n\lambda) \\
        &\hspace{5mm} \times \frac{\prod_{j=1}^n \left(p\left(r_j-\frac{k}{p-1}\right)\right)^{\nu(k,a_j)}}{\prod_{i=1}^{\hat{n}}\left(p\left(q_i-\frac{k}{p-1}\right)\right)^{\nu(k+1,b_i)}}\G_p\left(\frac{\ba-\frac{k}{p-1}, \bbeta}{\ba, \bbeta-\frac{k}{p-1}} \right).
    \end{singnumalign}
    As $\G_p$ takes values in the $p$-adic integers, we observe that the $k^{th}$ coefficient of $H_p$ will vanish modulo $p^2$ for all $k > a_2$.  Utilizing \Cref{thm:p-gamma-approx}, we find
    \begin{align*}
        \G_p\left(\frac{\ba-\frac{k}{p-1}, \bbeta}{\ba, \bbeta-\frac{k}{p-1}} \right) &=  \G_p\left(\frac{\ba+k+\frac{kp}{1-p},\bbeta}{\ba, \bbeta+k+\frac{kp}{1-p}}\right)\\
        &\equiv \G_p\left(\frac{\ba+k,\bbeta}{\ba,\bbeta+k}\right) \left(1+G_1(\ba+k, \bbeta+k)\frac{kp}{1-p}\right) \pmod{p^2}.
    \end{align*}
    Furthermore, by \eqref{eq:poch-p-gamma} we have
    \[
        \frac{p^{\nu(k,a_j)}\Gamma_p\left(r_j+k\right)}{\Gamma_p(r_j)} = (-1)^k(r_j)_k(r_j)^{-\nu(k,a_j)}, 
    \]
    and
    \[
        \frac{\Gamma_p(q_1)}{p^{\nu(k+1,b_i)}\Gamma_p(q_i+k)} = \frac{(-1)^kp^{\delta_{k=b_i}}(q_i)^{\nu(k,b_i)}}{(q_i)_k}
    \]
    with $\delta_{k=b_i} = 1$if $k=b_i$ and $0$ otherwise, and
    \[
        \Gamma_p(k+1) = (-1)^{k+1}(1)_k.
    \]
    Therefore,
    \begin{equation}\label{eq:GK-full-p-gamma}
        \G_p\left(\frac{\ba+k, \bbeta}{\ba, \bbeta+k}\right) p^{\sum_{j=1}^n \nu(k,a_j)- \sum_{i=1}^{\hat{n}} \nu(k+1,b_i)} = (-1)^{n-\hat{n}} \prod_{i=1}^n \frac{ (r_i)_k q_i^{\nu(k,b_i)}p^{\chi_{k=b_i}}}{r_i^{\nu(k,a_i)} (q_i)_k} .
    \end{equation}
    For $0 \leq k < b_1$, by definition \eqref{eq:x0-exp-defn} we have $x_i = \nu(b_i,k)-k+b_i = 1-k+b_i$, in which case the sign appearing in \eqref{eq:Hp-full-gamma-final} can be reduced as
    \[
        (-1)^{\hat{n}+\sum_{i=1}^{\hat{n}}(b_i+x_i)+(n-\hat{n})(k+1)} = (-1)^{n(k+1)+\hat{n}}.
    \]
    After factoring in the $(-1)^{n-\hat{n}}$ arising from \eqref{eq:GK-full-p-gamma} this will further reduce to $(-1)^{nk}$.  We let
    \begin{equation*}
        C(\ba, \bbeta,k) \colonequals  \left(\prod_{i=1}^n \frac{\left(1-\frac{k}{r_j(p-1)}\right)^{\nu(k,a_j)}p^{\chi_{k=b_i}}}{\left(1-\frac{k}{q_i(p-1)}\right)^{\nu(k,b_i)}}\right) \hypcoeff{k} \left(1 + G_1(\ba+k, \bbeta+k) \frac{kp}{1-p}\right),
    \end{equation*}
    so that
    \begin{equation}\label{eq:Hp-rewrite-short}
      H_p(\ba, \bbeta; \l; \bar{\omega}_p) \equiv \frac{1}{1-p} \sum_{k=0}^{a_2} (-1)^{nk}\omega_p^k((-1)^n\lambda)  C(\ba, \bbeta,k) \pmod{p^2}
    \end{equation}

    For $0 \leq k \leq a_1$, we have
    \begin{singnumalign}\label{eq:C-first-interval}
        C(\ba, \bbeta, k) &= \hypcoeff{k} \left(1+G_1(\ba+k, \bbeta+k) \frac{kp}{1-p}\right) \\
        &\equiv \hypcoeff{k} \left(1+G_1(\ba+k, \bbeta+k) kp\right) \pmod{p^2}
    \end{singnumalign}
    where the final congruence holds by using the fact that $1/(1-p) \equiv 1+p \pmod{p^2}$.  When $a_1 + 1 \leq k \leq a_2$, we instead have
    \begin{singnumalign}\label{eq:C-second-interval}
        C(\ba, \bbeta,k) &= \hypcoeff{k} \left(1+\frac{k}{r_1(1-p)}\right)\left(1+G_1(\ba+k, \bbeta+k) \frac{kp}{1-p}\right) \\
        &\equiv \hypcoeff{k}\left(1+\frac{k}{r_1}\right) \pmod{p^2}.
    \end{singnumalign}
    Together \eqref{eq:C-first-interval} and \eqref{eq:C-second-interval} allow us to expand \eqref{eq:Hp-rewrite-short} as
    \begin{align*}
        H_p(\ba, \bbeta; \l; \bar{\omega}_p) \equiv& \frac{1}{1-p} \biggr( \sum_{k=0}^{a_1} (-1)^{kn} \omega_p^k\left((-1)^n\lambda\right) \hypcoeff{k} \left(1 + G_1(\ba+k, \bbeta+k) kp\right) \\
        &\hspace{10mm} + \sum_{k=a_1+1}^{a_2} (-1)^{kn} \omega_p^k\left((-1)^n\lambda\right) \hypcoeff{k} \left(1+\frac{k}{r_1}\right)\biggr) \pmod{p^2}.
    \end{align*}
    Hensel lifting applied to the polynomial $f(x) = x^p-p$ with root $(-1)^n\lambda$ gives 
    \[
        \omega_p^k\left((-1)^n \lambda\right) \equiv (-1)^{nk} \lambda^{pk} \pmod{p^2}.
    \]
This together with the congruence $1/(1-p) \equiv 1+p \pmod{p^2}$ yields
    \begin{align*}
        H_p(\ba, \bbeta; \l; \bar{\omega}_p) \equiv& \sum_{k=0}^{a_1} \hypcoeff{k} (\lambda^p)^k \left(1 + (G_1(\ba+k, \bbeta+k)k+1)p\right) \\
        &+ \sum_{k=a_1+1}^{a_2} \hypcoeff{k} (\lambda^p)^k \left(1+\frac{k}{r_1}\right) \pmod{p^2}.
    \end{align*}
   This can now be decomposed as 
    \begin{align*}
        H_p(\ba, \bbeta; \l; \bar{\omega}_p) \equiv& \sum_{k=0}^{p-1} \hypcoeff{k} (\lambda^p)^k + \sum_{k=0}^{a_1} \hypcoeff{k} (\lambda^p)^k p(G_1(\ba+k,\bbeta+k)k+1) \\
        &+ \sum_{k=a_1+1}^{a_2} \hypcoeff{k} (\lambda^p)^k \frac{k}{r_1} \pmod{p^2}.
    \end{align*}
    The right-hand side is exactly $F(\ba, \bbeta; \l^p)_{p-1} + pE_{\mathrm{GK}}(\ba, \bbeta; \l^p)$, completing the proof.
\end{proof}
\subsubsection{The \texorpdfstring{$\l=1$}{lambda = 1} case}
Once again, the error term $E_{\mathrm{GK}}$ at $\lambda = 1$ can be computed explicitly via residue sums.  In particular, \Cref{rem:geom-error-terms} indicates that $E_{\mathrm{GK}}$ will correspond to the residue sum of $tR(t)$.
\begin{proposition}\label{prop:GK-error}
    With $\ba$, $\bbeta, M$, and $p$ as in \Cref{lem:beta-hat-residues} and 
    $R(t)$ as in \eqref{eq:Rt-defn},
    \[
        E_{\mathrm{GK}}(\ba, \bbeta; 1) \equiv (-1)^{1+\sum_{i=2}^{\hat{n}}b_i} \G_p\left(\frac{\bbeta}{\ba}\right) \mathrm{Res}_{t=\infty}(tR(t)) \pmod{p}.
    \]
\end{proposition}
\begin{proof}
    By definition \eqref{eq:E(GK)},
    \[
        E_{\mathrm{GK}}(\ba, \bbeta; 1) = \sum_{k=0}^{a_1}  \hypcoeff{k}  \left(kG_1(\boldsymbol{\alpha}+k, \boldsymbol{\bbeta}+k)+1\right) + \sum_{k=a_1+1}^{a_2}  \hypcoeff{k}  \frac{k}{pr_1}.
    \]
    In the same way that our previous error term corresponded to the residues of the function $R(t)$ defined in \eqref{eq:Rt-defn}, this error term will relate to the residues of $tR(t)$.  As before, we have a partial fraction decomposition of the form
    \begin{singnumalign}\label{eq:tRt-partial}
        tR(t) =& \sum_{k=0}^{a_1} \frac{\tilde{A}_k}{(t+k)^2} + \frac{\tilde{B}_k}{(t+k)} + \sum_{k=a_1+1}^{b_1} \frac{\tilde{B}_k}{t+k} + \sum_{k=b_1+1}^{p-1} \frac{\tilde{A}_k}{(t+k)^2} + \frac{\tilde{B}_k}{(t+k)} \\
        &+ \sum_{i=2}^{\hat{n}} \sum_{k=1}^{p-b_i-1} \frac{\tilde{C}_{i,k}}{-t+k+ip}
    \end{singnumalign}
    We use the intervals $I_1, I_2, I_3$, and $I_4$ defined in \eqref{eq:interval-defn}.  The coefficients $\tilde{A}_k$, $\tilde{B}_k$, and $\tilde{C}_{i,k}$ are closely related to the coefficients $A_k$, $B_k$, and $C_{i,k}$ appearing in \eqref{eq:R-partial}.  Specifically, from \Cref{lem:beta-hat-residues},
    \begin{equation}\label{eq:C-tilde}
        \tilde{C}_{i,k} = \lim_{t \to k+ip}(-t+k+ip)tR(t) = (k+ip)C_{i,k} \equiv 0 \pmod{p}.
    \end{equation}
    Similarly, the proof of \Cref{lem:I3-residues} implies that, for $k \in I_4$,
    \begin{singnumalign}\label{eq:B-tilde-I3}
        \tilde{B}_k 
        &= -kA_k \lim_{t \to -k} \frac{\frac{d}{dt}\left((t+k)^2tR(t)\right)}{(t+k)^2tR(t)} \equiv 0 \pmod{p}.
    \end{singnumalign}
    For $k \in I_1$, recalling \eqref{eq:Bk-Ak-ratio-I1} we have
    \begin{equation}
        \tilde{A}_k = -kA_k
    \end{equation}
    and
    \begin{equation}
        \tilde{B_k}/\tilde{A_k} = \frac{-1}{k} - G_1(\ba+k, \bbeta+k).
    \end{equation}
    For $k \in I_2 \cup I_3$ we have $\tilde{B_k} = -kB_k$.  In conclusion,
    \begin{equation*}\label{eq:Bk-tilde}
        \tilde{B_k} \equiv
        \begin{cases}
            (-1)^{\sum_{i=2}^{\hat{n}}b_i} \G_p\left(\frac{\ba}{\bbeta}\right)\hypcoeff{k} \left(1+kG_1(\ba+k, \bbeta+k)\right) & 0 \leq k \leq a_1 \\
            (-1)^{\sum_{i=2}^{\hat{n}}b_i} \frac{k}{pr_1}\G_p\left(\frac{\ba}{\bbeta}\right) \hypcoeff{k} & a_1+1 \leq k \leq a_2 \\
            0 & a_2+1 \leq k \leq p-1
        \end{cases}
        \pmod{p}.
    \end{equation*}
    Therefore,
    \begin{align*}
        \mathrm{Res}_{t=\infty}(tR(t)) &= -\sum_{k=0}^{p-1} \tilde{B_k} - \sum_{i=2}^{\hat{n}} \sum_{k=0}^{p-b_i-1} \tilde{C}_{i,k} \\&\equiv (-1)^{1+\sum_{i=2}^{\hat{n}}b_i} \G_p\left(\frac{\ba}{\bbeta}\right)E_{\mathrm{GK}}(\ba,\bbeta;1) \pmod{p}\end{align*}
    as was to be shown.
\end{proof}
In particular, the residue of $tR(t)$ at infinity is zero whenever the denominator has a degree at least two greater than the numerator, which gives us the following:
\begin{corollary}\label{cor:GK-super}
    Let $\HD = \left\{\ba, \bbeta\right\}$ be as in \Cref{lem:Dwork-congruence} satisfying $\gamma(\HD) < 1$.  For all primes $p\equiv 1\pmod{M(\HD)}$ such that $p> \hat{n}+1$, where as before $\hat{n}$ is the number of elements of $\bbeta$ not equal to $1$,
    \[
        H_p(\ba, \bbeta; \l; \bar{\omega}_p) \equiv F (\ba, \bbeta; 1)_{p-1} \pmod{p^2}
    \]
    If instead $\gamma(\HD)=1$, from the definition of $R(t)$ given by \eqref{eq:Rt-defn}, we have
    \[
        \mathrm{Res}_{t=\infty}(tR(t))=-\lim_{|t| \to \infty} (t^2R(t)) = (-1)^{1+\sum_{i=2}^{\hat{n}}b_i}.
    \]
    Thus, by \Cref{prop:GK-error}, 
    \begin{equation*}
        E_{\textit{GK}}(\ba, \bbeta; 1) \equiv  \G_p\left(\frac{\bbeta}{\ba}\right)  \pmod{p},
    \end{equation*}
    so that \Cref{lem:GK-congruence} becomes 
    \begin{equation*}\label{eq:GK-residue-1}
        H_p(\ba, \bbeta; \l; \bar{\omega}_p) \equiv F (\ba, \bbeta; 1)_{p-1} + \G_p\left(\frac{\bbeta}{\ba}\right)p \pmod{p^2}.
    \end{equation*}
\end{corollary}
We conclude this section by deducing \Cref{thm:supercongruences}.
\begin{proof}[Proof of \Cref{thm:supercongruences}]
    As we are assuming $\gamma(\HD)\le 1$, \Cref{cor:GK-super} yields
    \[
        H_p(\ba, \bbeta; 1; \bar{\omega}_p) -\delta_{\gamma(\HD)=1}\G_p\left(\frac{\bbeta}{\ba}\right)p \equiv F(\ba, \bbeta; 1)_{p-1} \pmod{p^2}.
    \]
    By \Cref{cor:Dwork-super}, it follows that for all positive integers $s$,
    \[
        H_p(\ba, \bbeta; 1; \bar{\omega}_p) -  \delta_{\gamma(\HD)=1}\G_p\left(\frac{\bbeta}{\ba}\right)p \equiv \frac{F_{s+1}(\ba, \bbeta;1)}{F_s(\ba, \bbeta;1)} \pmod{p^2}.
    \]
    Letting $s$ go to infinity and using Dwork's congruence \eqref{eq:Dwork-unit-root-congruence} then gives the result.
\end{proof}
In the ensuing discussion, we will primarily focus on cases where the hypergeometric data $\HD$ has length three or four. As such, we state the specializations of \Cref{cor:GK-super} to these settings.  For the length four case, we utilize \Cref{cor:Jac-Gp} to express our supercongruence using $\P$ rather than $H_p$. 
%

\section{\texorpdfstring{Proof of \Cref{thm:main}}{Proof of Theorem 2.1}} \label{sec:hgGalois}

We are now ready to prove \Cref{thm:main}. 
\begin{proof}[Proof of \Cref{thm:main}]

The length $n$ of $\HD$ is 3 or 4 by assumption. By \Cref{thm:Katz}, the representation $\rho_{\{HD;1\}}$ of $G(M)$ is $n-1$ dimensional. Condition (2) of \Cref{thm:main} implies $\chi_{\HD}\otimes\rho_{\{\HD;1\}}$ of $G(M)$ has an  extension  to $G_\Q$; use $\hat \rho_{\{HD;1\}}$ to denote such an extension as in \Cref{prop:4.3}.  When $n=4$, condition (1) of \Cref{thm:main} implies the three-dimensional representation $\chi_{\HD}\otimes\rho_{\{HD;1\}}$ is reducible so $\gamma(\HD)=1$, see Remark \ref{remakr:length4}. Part (iii) of \Cref{thm:Katz} implies if this representation is self-dual, (which is the case by condition (2)), then it contains a nontrivial $d$-dimensional subrepresentation space that admits a non-degenerate alternating pairing. Thus $d$ is even, and so must be 2. Then $\chi_{\HD}\otimes\rho_{\{\HD;1\}}$ is the direct sum of two sub-representations both extendable to $G_\Q$.  Namely $
    \hat \rho_{\{HD;1\}}=\hat \rho_{\{HD;1\}}^{prim}\oplus \left(\psi_{\HD}\otimes \epsilon_\ell|_{G(M)}\right)
$ where $\hat \rho_{\{HD;1\}}^{prim}$ is 2-dimensional, $\psi_{\HD}$ is a finite order character of $G(M)$ such that $\psi_{\HD}(\wp)=\pm 1$ for any prime ideal $\wp$ of $\Z[\zeta_M]$ above $p\in P_M=\{p\mid p \text{ prime, } p\equiv 1\pmod M\}$, and $\epsilon_\ell$ denotes the $\ell$-adic cyclotomic character. When $n=3$, we refer to $\hat \rho_{\{HD;1\}}$ as $\hat \rho_{\{HD;1\}}^{prim}$ for uniformity of notation.  In either case, for any prime ideal $\wp$ above $p\in P_M$, \begin{equation}\label{eq:extension-decomp}
   \Tr \hat \rho_{\{HD;1\}}(\Frob_\wp)=\Tr\hat \rho_{\{HD;1\}}^{prim}(\Frob_\wp)+ \delta_{\gamma(\HD)=1}\psi_{\HD}(p)p.
\end{equation}
 We use $p$ rather than $\mathfrak{p}$ in $\psi_{\HD}$ as the character is independent of choice of $\wp$ over $p$.  \medskip
 
 We now relate the discussion to the modular target $f_{\HD}^\sharp$. By the assumption (1) of \Cref{thm:main}, if $f_{\HD}$ is inside a Hecke orbit then any Hecke eigenform in the orbit is a natural candidate for $f_{\HD}^\sharp$. Further, assumption (2) of \Cref{thm:main} implies the well-definedness of  $f_{\HD}^\sharp$ which will work for not only $\HD$ but also its conjugates, as in \S \ref{ss:K2}. Next, we will show that Deligne's representation $\rho_{f_{\HD}^\sharp}|_{G(M)}\simeq \hat \rho_{\{HD;1\}}^{prim}.$ By the CFGL discussion in \Cref{ss:CFGL} amounting to \Cref{prop:CFGL}, we know for $p\equiv 1\pmod M$ their Frobenius traces agree modulo $p$ and have the same unit root in the ordinary case.  To proceed we show that the supercongruences in \Cref{thm:supercongruences} are enough to strengthen this from a congruence modulo $p$ to equality.

By part (iii) of \Cref{thm:Katz}, the weight of $f_{\HD}^\sharp$ is $n$. By condition (2) of \Cref{thm:main} and the discussion in \Cref{ss:CFGL}, $a_p(f^\sharp_{\HD})=\tilde b_p$ for $p\in P_M$, where $\tilde b_p$ is the $p$th Hecke eigenvalue of $f_{\HD}$, as in the statement of \Cref{thm:main}. For any such prime ideal $\wp$ that is ordinary, 
\begin{equation}\label{eq:ap-u}
\Tr \hat \rho_{\{HD;1\}}^{prim}(\Frob_\wp)=a_p(f_{\HD}^\sharp)=u_{f_{\HD},p}+ \varphi(p)p^{n-1}/u_{f_{\HD},p},    
\end{equation} where  $u_{f_{\HD},p}\in\Z_p^\times$ is a solution of $T^2- a_p(f_{\HD}^\sharp)T+ \varphi(p)p^{n-1}=0$   
which can be also computed by \eqref{eq:u_f} using CFGL. \medskip

We turn our attention to the traces  of the Katz representation. By condition (2) of \Cref{thm:main}, for every nonzero prime ideal $\wp$ above $p\in P_M$, $\chi_{\HD}(\wp)\P(\HD;1;\wp)$ is  in $\Z$, where we recall $\chi_{\HD}(\wp)
=\iota_\wp(r_n)(C_1)^{-1}\cdot \prod_{i=1}^{n-1}\iota_\wp(r_i)(-1)$.
To proceed, we fix the $p$-adic embedding  $\iota_\wp(\frac{1}{p-1})\mapsto \bar \omega_p$ as in \eqref{eq:Hp-embedded}. Thus, the  embedding of $\iota_\wp(r_n)(C_1)^{-1}$ into $\Z_p$ is $\omega_p^{(p-1)r_n}(C_1)$. Together with  the conversion from $\P$ to $H$ given by \eqref{eq:P-H},  the  embedding of $\Tr \hat \rho_{\{HD;1\}}(\Frob_\wp)=(-1)^{n-1}\chi_{\HD}(\wp)\P( {\HD};1;\wp)$ is $\omega_p^{(p-1)r_n}(C_1)\G_p\left(\frac{r_n,q_n-r_n}{q_n}\right)H_p( {\HD};1;\bar \omega_p)$. Comparing with the earlier formula for $\Tr \hat \rho_{\{HD;1\}}(\Frob_\wp)$ in \eqref{eq:extension-decomp} we have
\begin{eqnarray*}
\omega_p^{(p-1)r_n}(C_1)\G_p\left(\frac{r_n,q_n-r_n}{q_n}\right)H_p( {\HD};1;\bar \omega_p)=
a_p(f_{\HD}^\sharp)+\delta_{\gamma(\HD)=1}\cdot\psi_{\HD}(p)\cdot p,
\end{eqnarray*} which is \eqref{eq:2.9}. Note that modulo $p$, the above is equivalent to \eqref{eq:cong-mod-p}. 
\medskip

What remains to be determined is the sign $\psi_{\HD}(p),$ which follows from \Cref{thm:supercongruences}. Applying \eqref{eq:super-combined} to $H_p( {\HD};1;\bar \omega_p)$, one has
\begin{multline}\label{eq:6.3}
a_p(f_{\HD}^\sharp)+\delta_{\gamma(\HD)=1}\cdot\psi_{\HD}(p)\cdot p\equiv  \omega_p^{(p-1)r_n}(C_1)\G_p\left(\frac{r_n,q_n-r_n}{q_n}\right)\\ \times \left(  F(\ba, \bbeta; 1)_{p-1}+\delta_{\gamma(\HD)=1}\G_p\left(\frac{\bbeta}{\ba}\right) p\right) \pmod{p^2}.
\end{multline}
From the comparison, at ordinary primes $p$, we have the following agreement of unit roots. $$\mu_{f_{\HD},p}=\omega_p^{(p-1)r_n}(C_1)\G_p\left(\frac{r_n,q_n-r_n}{q_n}\right)\mu_{\HD,1,p}$$

When $\gamma(\HD)=1$, by comparing the terms of \eqref{eq:6.3} with $p$-adic valuation one at ordinary $p\in P_M$, one has

\begin{equation}\label{eq:(6.4)}
\begin{split}
 \psi_{\HD}(p)&\equiv\omega_p^{(p-1)r_n}(C_1)\G_p\left(\frac{r_n,q_n-r_n}{q_n}\right)\G_p\left (\frac{\bbeta}{\ba}\right)\\&\equiv \omega_p^{(p-1)r_n}(C_1)\G_p\left(\frac{\bbeta^\flat,q_n-r_n}{\ba^\flat}\right)\\&\equiv (-1)^{n-1} C_1^{(p-1)r_n}\G_p\left(\frac{q_n-r_n}{\ba^\flat}\right)\pmod p, 
\end{split}
\end{equation}as $(-1)^{n-1}=\G_p(\bbeta^\flat)=\G_p(1)^{(n-1)}$. Thus the claim \eqref{eq:sgn} is established. \medskip

To finish the proof, we note that for $\wp$ above $p\equiv 1\pmod M$,  the absolute value of the integer $$ (-1)^{n-1}\chi_{\HD}(\wp)\P(\HD;1;\wp)-\delta_{\gamma(\HD)=1}\cdot \psi_{\HD}(p)\cdot p$$  is less than $3p^{3/2}+p$ when $n=4$ (resp. $2p$ when $n=3$) while $a_p(f_{\HD}^\sharp)$ is an integer whose absolute value is less than $2p^{3/2}$ (resp. $2p$). Theorem \ref{thm:supercongruences} implies that they agree modulo $p^2$. For all $p \geq 29$, $5p^{3/2} + p < p^2$, and so the two integers must be the same.  Similarly, for $n=3$ they agree for all $p\ge 5$.
\end{proof}

\begin{remark}
    From the modularity of two-dimensional motivic $G_\Q$ representations---see for example Theorem 1.0.4 of \cite{PanLue2022}--- any extension of $\hat \rho_{\{HD;1\}}^{prim}$ to $G_\Q$ is isomorphic to a Deligne representation associated with a Hecke eigenform $f_{\HD}^\sharp$ of character $\varphi=\varphi_{\HD}$. Our method gives the explicit Deligne representation by constructing the corresponding Hecke eigenforms which are unique up to twisting. 
\end{remark}

We end this section by pointing out that the Galois condition given by \Cref{defn:galois} often gives rise to Galois representations being extendable to $G_\Q$. Here we give a demonstration for the case discussed in \Cref{prop:K(j/8)}.

\begin{proposition}\label{eq:Prop1/8}
For each $j\in\{1,3,5,7\}$, any prime $p\equiv 1\pmod{8}$, and any prime $\wp\in\Z[\zeta_8]$ above $p$
    $$  \P\left(\left\{\frac12,\frac12,\frac j8\right\},\{1,1,1\};1;\wp\right)=a_p\left (f_{256.3.c.g}\right).$$ 
\end{proposition}
\begin{proof} Note that under the assumptions, $\chi_{\HD}(\wp)=1$.
First note that the right-hand side is an integer whose absolute value is less than or equal to $2p$ from the Weil bound. Meanwhile, the left-hand side takes value in $\Z[\zeta_8]$, so it can be written as $a_{11}+a_{12}\sqrt{-1}+a_{21}\sqrt{2}+a_{22}\sqrt{-2}$ for $a_{ij}\in \frac12\Z$. There are four different ways to embed $\Z[\zeta_8]$ in $\Q_p$, namely if we fix one embedding $\sigma(\zeta_8)\in\Q_p$, then all such embeddings  are given by sending $\zeta_8\mapsto \sigma(\zeta_8)^j$ where $j=1,3,5,7$. We note them by $\sigma_j$ accordingly. By \eqref{eq:3.21} and the corresponding modulo $p^2$ supercongruences we know for each $j$
$$\sigma_j\left(a_{11}-a_p\left (f_{256.3.c.g}\right)+a_{12}\sqrt{-1}+a_{21}\sqrt{2}+a_{22}\sqrt{-2}\right)\equiv 0 \pmod {p^2}.$$ Thus $p^2$ divides each of the half integers $a_{11}-a_p\left (f_{256.3.c.g}\right), a_{12},a_{21},a_{22}$. If any of these half integers is nonzero, then any complex norm of 

$$ \P\left(\left\{\frac12,\frac12,\frac 18\right\},\{1,1,1\};1;\wp\right)-a_p\left (f_{256.3.c.g}\right)$$ is at least $p^2/2$, while any complex norm of either   $  \P\left(\left\{\frac12,\frac12,\frac 18\right\},\{1,1,1\};1;\wp\right)$ or $a_p\left (f_{256.3.c.g}\right)$   is less than or equal to $2p,$ this is impossible when $p\ge 17$. Thus the difference between both hand sides is 0.
\end{proof}

The $j = 1$ case of Proposition \ref{eq:Prop1/8} resolves case $3$ of the conjectures of Dawsey and McCarthy in Table 2 of \cite{Dawsey-McCarthy}.


\section{Appendix: 
Modular Forms on Triangle Groups} \label{sec:triangle} 
A crucial part of our method is computing the target modular form explicitly using the integral representation of classical hypergeometric functions and an appropriate Hauptmodul. In the $n = 3$ case, the values of a classical $_{2}F_{1}$ function evaluated at a Hauptmodul $t_{d}$ are used.   In \cite{Stiller85}, Stiller pointed out that suitable homogeneous solutions of certain degree-2 ordinary differential equations can be viewed as weight-1  forms of the corresponding monodromy groups. Conversely, it is known that weight-1 (resp. more generally $k$)  modular forms are annihilated by  degree-2 (resp. $k+1$) ordinary differential equations in terms of non-constant modular functions, see \cite{Yang04, Yang-Schwarzian} by Yang.  The Schwarz map, mentioned earlier relates the $_{2}F_{1}$ function in question to a triangle group, that is commensurable to $\text{SL}_{2}(\Z)$, in the cases we consider.

A similar idea is used in the $n = 4$ case. However, here the Clausen formula is needed to relate the target $_{3}F_{2}$ function, which now appears in the computation, back to a square of $_{2}F_{1}$ functions.

Below we list a few relevant groups commensurable with $\SL_2(\Z)$ related to our method, given in the partial order of containment, see Figure \ref{fig:1}.  

\begin{figure}[ht]
    \centering
    \[
        \xymatrix{
        &\genfrac{}{}{0pt}{1}{(2,6,\infty)}{\G_0^+(3);\,t_{6+}}\ar@{-}[dr]^{2}&& \genfrac{}{}{0pt}{1}{(2,3,\infty)}{\SL_2(\Z);\, \frac{1728}{j}} \ar@{-}[dl]^{4}\ar@{-}[dr]^{3} && \genfrac{}{}{0pt}{1}{(2,4,\infty)}{\G_0^+(2);\,  t_{4+}}\ar@{-}[dl]^{2} \\
        &&   \genfrac{}{}{0pt}{1}{(3,\infty,\infty)}{\G_0(3); \, t_3}&& \genfrac{}{}{0pt}{1}{(2,\infty,\infty)}{\G_0(2); \, t_4}\ar@{-}[d]^{2} \\
        &&&& \genfrac{}{}{0pt}{1}{(\infty,\infty,\infty)}{\G(2)\simeq \G_0(4); \, t_2}
        }
    \]
    \caption{Some triangle groups commensurable with $\SL_2(\Z)$ }
    \label{fig:1}
\end{figure}

For each group $\G$ in Figure \ref{fig:1}, we choose a Hauptmodul $t(\tau)$ which takes
the values $0$, $1$, and $\infty$ at the elliptic points $e_1,e_2,e_3$ of $\G$. By abuse of notation, we regard each cusp as an elliptic point of infinite order. See Table \ref{tab:Hauptmodul} below for the exact functions. The first $t(\tau)$ coefficient, denoted as $C_{1}$, is important in our explicit method.
\begin{table}[ht]
 \centering
 \begin{tblr}{
 width = \textwidth,
 colspec ={||X[c,1.2]||X[c,5.2]|X[c,1]|X[c,2.6]|X[c,2]|X[c,1.9]||},
 }  \hline
    \mbox{Group} &  \mbox{Hauptmodul}\,  $t$ & $C_1$ & ${e_1,e_2,e_3}$ &\mbox{Orders}& $t \overset{ W}\leftrightarrow 1-t$\\ \hline
    $\G_0(1)$ & $1728/j(\tau)$ & $1728$ & $i\infty, i, \frac{-1+i\sqrt{3}}{2}$  & $\infty, 2, 3$ &\\ \hline 
    $\G(2)$ & $\l(\tau)\colonequals16 \frac{\eta(\frac{\tau}{2})^{8}\eta(2 \tau)^{16}}{\eta(\tau)^{24}}$ & $16$ & $i\infty, 0, 1$ & $\infty, \infty, \infty$ & $\M 0{-1}10$ \\ \hline 
    $\G_0(3)$ &  $t_{3}(\tau)\colonequals\frac{27\eta(3 \tau)^{12}}{\eta(\tau)^{12} + 27\eta(3 \tau)^{12}}$ & $27$ & $i\infty, 0,  \frac{3+i\sqrt 3}{6}$ & $\infty, \infty, 3$ & $\M 0{-1}30$ \\ \hline 
    $\G_0(4)$ & $t_2(\tau)\colonequals\l(2\tau)$ & $16$ & $i\infty, 0, \frac 12$ & $\infty, \infty, \infty$ &\\ \hline 
    $\G_0(2)$ &  $u(\tau)\colonequals-64\frac{\eta(2\tau)^{24}}{\eta(\tau)^{24}}$ & $-64$ & $i\infty , \frac{1+i}{2}, 0$ & $\infty, 2, \infty$ & $\M 0{-1}20$ \\ \hline 
    $\G_0(2)^+$ & $t_{4+}(\tau)\colonequals\frac{256\eta(\tau)^{24}\eta(2\tau)^{24}}{(\eta(\tau)^{24}+64\eta(2\tau)^{24})^2}$ & $256$ & $i\infty, \frac{i}{\sqrt 2}, \frac{1+i}{2}$ & $\infty, 2,4$ &\\ \hline 
    $\G_0(3)^+$ & $t_{6+}(\tau)\colonequals \frac{ 108\eta(\tau)^{12}\eta(3\tau)^{12}}{(\eta(\tau)^{12}+27\eta(3\tau)^{12})^2}$ & $108$ & $i\infty, \frac{i}{\sqrt 3}, \frac{3+i\sqrt 3}6$ & $\infty, 2, 6$ & \\\hline
 \end{tblr}
 \caption{Hauptmoduln}\label{tab:Hauptmodul}
\end{table}
We also give explicit evaluations of certain hypergeometric functions evaluated at these Hauptmoduln. 
We set
\[
    \ba^{(2)}_d = \left\{\frac{1}{d}, 1-\frac{1}{d}\right\}, \qquad \ba^{(3)}_d = \left\{ \frac{1}{2}, \frac{1}{d}, 1-\frac{1}{d}\right\}.
\]
Additionally, we let ${}_2F_1(\ba^{(2)}_d; t_d)$ and ${}_3F_2(\ba^{(3)}_d; t_{d+})$ denote 
\[
    \pFq21{\, \,\frac1d&\frac{d-1}{d}}{\, \,1&1}{t_d} \quad \text{and} \quad \pFq32{\, \frac 12&\frac1d&\frac{d-1}{d}}{\, \, 1&1&1}{t_{d+}},
\]
respectively.  The $_{2}F_{1}(\ba^{(2)}_d; t_d)$ value for $d = 2$ is given in \eqref{eqn: alt-2}. The analogous theory for $d = 3$ involves the Hauptmodul $t_{3}$ from \Cref{tab:Hauptmodul}.
In \cite{BBG}, Borwein, Borwein, and Garvan show that
\begin{equation}\label{eqn: alt-3}
   \pFq21{\, \,\frac13&\frac23}{\, \, 1&1}{t_3(\tau)}=\sum_{n,m\in \Z}q^{n^2+nm+m^2},
\end{equation}
when both sides converge. Another weight $\frac{1}{2}$ Jacobi theta function which appears in the $_{3}F_{2}(\ba_{d}^{(3)};t_{d+})$ function when $d = 2$ is $\theta_{4}$. The definition of $\theta_{4}$ and the connection to $\theta_{3}$ are given in \eqref{eq:theta4}. Now combining \eqref{eqn: alt-2} and \eqref{eqn: alt-3} with the Clausen formula for hypergeometric functions gives the following information
\[
\begin{tblr}{
width=0.45\textwidth,
colspec={|c|c|c|c|},
}\hline
 d&\G_d & t_d  & {}_2F_1(\ba^{(2)}_d; t_d)\\
 \hline
 2 &\G(2) & \l(\tau)&  \theta_3(\tau)^2\\ \hline
 2 &\G_0(4)&\l(2\tau)& \theta_3(2\tau)^2\\ \hline
 3 & \G_0(3)  & t_3(\tau)   & \frac{3 \eta(3\tau)^{3} + \eta\left(\frac{\tau}{3}\right)^{3}}{\eta(\tau)} \\ \hline
 4 & \G_0(2)& \frac{u(\tau)}{u(\tau)-1} &(-E_{2,2}(\tau))^{1/2}  \\ \hline
\end{tblr}
\qquad
\begin{tblr}{
 width=0.45\textwidth,
 colspec ={|c|c|c|c|c|c|c|},
 }\hline
 d&\G_d & t_{+} & {}_3F_2(\ba^{(3)}_d; t_{d+}) \\\hline
 2 &\G_0(2) &  u(\tau) & \theta_4(2\tau)^4 \\ \hline
 3 & \G_0(3)^+  &  t_{6+}(\tau) &-\frac{1}{2} E_{2,3}(\tau) \\ \hline
 4 & \G_0(2)^+&  t_{4+}(\tau) & -E_{2,2}(\tau)  \\ \hline
 6&  \G_0(1)& \frac{1728}{j(\tau)} & E_4(\tau)^{1/2}  \\ \hline
\end{tblr}
\label{tab:data}
\]
where  $E_{2,N}(\tau)\colonequals E_2(\tau)-NE_2(N\tau)$, $\displaystyle E_2(\tau)=1-24\sum_{n\ge 1}\frac{nq^n}{1-q^n}$ is the weight two Eisenstein series---a holomorphic quasi-modular form, and $E_{4}$ is the normalized weight four Eisenstein series on $\SL_{2}(\Z)$. We also remark that when $d=6$,  by a hypergeometric quadratic formula,  
   $${}_2F_1(\ba^{(2)}_6; t_6)= \pFq21{\, \, \frac1{12}&\frac5{12}}{\, \,1&1}{4t_6(1-t_6)}= \pFq21{\, \, \frac1{12}&\frac5{12}}{\, \, 1&1}{\frac{1728}{j}} =E_4^{1/4},$$
where $t_6$ is a function satisfying $4t_6(1-t_6)=\frac{1728}{j}$.  

The proof of \Cref{lem:E-eta} indicates understanding the logarithmic derivatives of the Hauptmoduln of interest is an important part of computing the target modular form. Therefore, the logarithmic derivatives for the Hauptmoduln of interest are recorded below. 
\begin{proposition}\label{prop:LogDerivative}
Use $\Theta\colonequals q\frac{d}{dq}\log $ to denote the logarithmic derivative in terms of $q=e^{2\pi i \tau}$, then
  \begin{align*}
      \Theta \l(\tau) &= \frac 12\theta_4^4(\tau)\\
     \Theta u(\tau) &=  \frac{1+\l(2\tau)}{1-\l(2\tau)} \theta_4^4(2\tau)= (1-u)^{1/2} \theta_4^4(2\tau)\\
    \Theta   t_3(\tau)& = (1-t_3(\tau)) \left(\frac{3 \eta(3\tau)^{3} + \eta\left(\frac{\tau}{3}\right)^{3}}{\eta(\tau)}\right)^2 \\
     \Theta j(\tau) & = \frac{E_6(\tau)}{E_4(\tau)}\\    
    \Theta t_{4+}(\tau)& = -E_{2,2}(\tau) \frac{\eta(\tau)^{24}-64\eta(2\tau)^{24}}{\eta(\tau)^{24}+64\eta(2\tau)^{24}} \\
    \Theta t_{6+}(\tau)&= -\frac 12E_{2,3}(\tau) \frac{\eta(\tau)^{12}-27\eta(3\tau)^{12}}{\eta(\tau)^{12}+27\eta(3\tau)^{12}}
   \end{align*}   
\end{proposition}

\bibliographystyle{habbrv}
\bibliography{ref}

\end{document}